\documentclass[11pt]{amsart}
\usepackage{}
\usepackage{amssymb}
\usepackage{amsfonts}
\usepackage{mathrsfs}
\usepackage{latexsym}
\usepackage{graphicx}
\usepackage{amscd,amssymb,amsmath,amsbsy,amsthm,amsfonts}
\usepackage[all]{xy}
\usepackage[colorlinks,plainpages,backref,urlcolor=blue]{hyperref}
\usepackage{verbatim}
\usepackage{enumerate}

\usepackage [english]{babel}
\usepackage [autostyle, english = american]{csquotes}
\MakeOuterQuote{"}

\newcommand {\Omit}[1]{}

\usepackage{tikz-cd}

\usepackage{tikz}
\usetikzlibrary{automata,positioning}

\usetikzlibrary{shapes,shadows}

\usetikzlibrary{arrows,calc}
\tikzset{
>=stealth',
help lines/.style={dashed, thick},
axis/.style={<->},
important line/.style={thick},
connection/.style={thick, dotted},
}
\usetikzlibrary{patterns}
\newlength{\hatchspread}
\newlength{\hatchthickness}
\tikzset{hatchspread/.code={\setlength{\hatchspread}{#1}},
         hatchthickness/.code={\setlength{\hatchthickness}{#1}}}
\tikzset{hatchspread=3pt,
         hatchthickness=0.4pt}
\pgfdeclarepatternformonly[\hatchspread,\hatchthickness]
   {custom north west lines}
   {\pgfqpoint{-2\hatchthickness}{-2\hatchthickness}}
   {\pgfqpoint{\dimexpr\hatchspread+2\hatchthickness}{\dimexpr\hatchspread+2\hatchthickness}}
   {\pgfqpoint{\hatchspread}{\hatchspread}}
   {
    \pgfsetlinewidth{\hatchthickness}
    \pgfpathmoveto{\pgfqpoint{0pt}{\hatchspread}}
    \pgfpathlineto{\pgfqpoint{\dimexpr\hatchspread+0.15pt}{-0.15pt}}
    \pgfusepath{stroke}
   }

\newcommand{\nc}{\newcommand}
\nc{\rnc}{\renewcommand}
\nc{\bb}[1]{{\mathbb #1}}
\nc{\bbA}{\bb{A}}\nc{\bbB}{\bb{B}}\nc{\bbC}{\bb{C}}\nc{\bbD}{\bb{D}}
\nc{\bbE}{\bb{E}}\nc{\bbF}{\bb{F}}\nc{\bbG}{\bb{G}}\nc{\bbH}{\bb{H}}
\nc{\bbI}{\bb{I}}\nc{\bbJ}{\bb{J}}\nc{\bbK}{\bb{K}}\nc{\bbL}{\bb{L}}
\nc{\bbM}{\bb{M}}\nc{\bbN}{\bb{N}}\nc{\bbO}{\bb{O}}\nc{\bbP}{\bb{P}}
\nc{\bbQ}{\bb{Q}}\nc{\bbR}{\bb{R}}\nc{\bbS}{\bb{S}}\nc{\bbT}{\bb{T}}
\nc{\bbU}{\bb{U}}\nc{\bbV}{\bb{V}}\nc{\bbW}{\bb{W}}\nc{\bbX}{\bb{X}}
\nc{\bbY}{\bb{Y}}\nc{\bbZ}{\bb{Z}}
\nc{\mbf}[1]{{\mathbf #1}}
\nc{\bfA}{\mbf{A}}\nc{\bfB}{\mbf{B}}\nc{\bfC}{\mbf{C}}\nc{\bfD}{\mbf{D}}
\nc{\bfE}{\mbf{E}}\nc{\bfF}{\mbf{F}}\nc{\bfG}{\mbf{G}}\nc{\bfH}{\mbf{H}}
\nc{\bfI}{\mbf{I}}\nc{\bfJ}{\mbf{J}}\nc{\bfK}{\mbf{K}}\nc{\bfL}{\mbf{L}}
\nc{\bfM}{\mbf{M}}\nc{\bfN}{\mbf{N}}\nc{\bfO}{\mbf{O}}\nc{\bfP}{\mbf{P}}
\nc{\bfQ}{\mbf{Q}}\nc{\bfR}{\mbf{R}}\nc{\bfS}{\mbf{S}}\nc{\bfT}{\mbf{T}}
\nc{\bfU}{\mbf{U}}\nc{\bfV}{\mbf{V}}\nc{\bfW}{\mbf{W}}\nc{\bfX}{\mbf{X}}
\nc{\bfY}{\mbf{Y}}\nc{\bfZ}{\mbf{Z}}
\nc{\bfa}{\mbf{a}}\nc{\bfb}{\mbf{b}}\nc{\bfc}{\mbf{c}}\nc{\bfd}{\mbf{d}}
\nc{\bfe}{\mbf{e}}\nc{\bff}{\mbf{f}}\nc{\bfg}{\mbf{g}}\nc{\bfh}{\mbf{h}}
\nc{\bfi}{\mbf{i}}\nc{\bfj}{\mbf{j}}\nc{\bfk}{\mbf{k}}\nc{\bfl}{\mbf{l}}
\nc{\bfm}{\mbf{m}}\nc{\bfn}{\mbf{n}}\nc{\bfo}{\mbf{o}}\nc{\bfp}{\mbf{p}}
\nc{\bfq}{\mbf{q}}\nc{\bfr}{\mbf{r}}\nc{\bfs}{\mbf{s}}\nc{\bft}{\mbf{t}}
\nc{\bfu}{\mbf{u}}\nc{\bfv}{\mbf{v}}\nc{\bfw}{\mbf{w}}\nc{\bfx}{\mbf{x}}
\nc{\bfy}{\mbf{y}}\nc{\bfz}{\mbf{z}}

\newcommand{\op}{\text{op}}

\nc{\mcal}[1]{{\mathcal #1}}
\nc{\calA}{\mcal{A}}\nc{\calB}{\mcal{B}}\nc{\calC}{\mcal{C}}\nc{\calD}{\mcal{D}}
\nc{\calE}{\mcal{E}} \nc{\calF}{\mcal{F}}\nc{\calG}{\mcal{G}}\nc{\calH}{\mcal{H}}
\nc{\calI}{\mcal{I}}\nc{\calJ}{\mcal{J}}\nc{\calK}{\mcal{K}}\nc{\calL}{\mcal{L}}
\nc{\calM}{\mcal{M}}\nc{\calN}{\mcal{N}}\nc{\calO}{\mcal{O}}\nc{\calP}{\mcal{P}}
\nc{\calQ}{\mcal{Q}}\nc{\calR}{\mcal{R}}\nc{\calS}{\mcal{S}}\nc{\calT}{\mcal{T}}
\nc{\calU}{\mcal{U}}\nc{\calV}{\mcal{V}}\nc{\calW}{\mcal{W}}\nc{\calX}{\mcal{X}}
\nc{\calY}{\mcal{Y}}\nc{\calZ}{\mcal{Z}}
\nc{\fA}{\frak{A}}\nc{\fB}{\frak{B}}\nc{\fC}{\frak{C}} \nc{\fD}{\frak{D}}
\nc{\fE}{\frak{E}}\nc{\fF}{\frak{F}}\nc{\fG}{\frak{G}}\nc{\fH}{\frak{H}}
\nc{\fI}{\frak{I}}\nc{\fJ}{\frak{J}}\nc{\fK}{\frak{K}}\nc{\fL}{\frak{L}}
\nc{\fM}{\frak{M}}\nc{\fN}{\frak{N}}\nc{\fO}{\frak{O}}\nc{\fP}{\frak{P}}
\nc{\fQ}{\frak{Q}}\nc{\fR}{\frak{R}}\nc{\fS}{\frak{S}}\nc{\fT}{\frak{T}}
\nc{\fU}{\frak{U}}\nc{\fV}{\frak{V}}\nc{\fW}{\frak{W}}\nc{\fX}{\frak{X}}
\nc{\fY}{\frak{Y}}\nc{\fZ}{\frak{Z}}
\nc{\fa}{\frak{a}}\nc{\fb}{\frak{b}}\nc{\fc}{\frak{c}} \nc{\fd}{\frak{d}}
\nc{\fe}{\frak{e}}\nc{\fFf}{\frak{f}}\nc{\fg}{\frak{g}}\nc{\fh}{\frak{h}}
\nc{\fri}{\frak{i}}\nc{\fj}{\frak{j}}\nc{\fk}{\frak{k}}\nc{\fl}{\frak{l}}
\nc{\fm}{\frak{m}}\nc{\fn}{\frak{n}}\nc{\fo}{\frak{o}}\nc{\fp}{\frak{p}}
\nc{\fq}{\frak{q}}\nc{\fr}{\frak{r}}\nc{\fs}{\frak{s}}\nc{\ft}{\frak{t}}
\nc{\fu}{\frak{u}}\nc{\fv}{\frak{v}}\nc{\fw}{\frak{w}}\nc{\fx}{\frak{x}}
\nc{\fy}{\frak{y}}\nc{\fz}{\frak{z}}

\newtheorem{theorem}{Theorem}[section]
\newtheorem{lemma}[theorem]{Lemma}

\theoremstyle{definition}

\newtheorem{example}[theorem]{Example}
\newtheorem{remark}[theorem]{Remark}

 \DeclareMathOperator{\id}{id}

 \DeclareMathOperator{\GL}{GL}

\DeclareMathOperator{\Hom}{{Hom}} 

\DeclareMathOperator{\sHom}{{\mathscr{H}om}}

 \DeclareMathOperator{\Lie}{Lie}
 \DeclareMathOperator{\tr}{tr}

\DeclareMathOperator{\Grass}{Grass}

\DeclareMathOperator{\Gm}{\bbG_m}

   \DeclareMathOperator{\sph}{sph}

      \DeclareMathOperator{\crit}{crit}

\DeclareMathOperator{\diag}{diag}

\DeclareMathOperator{\Crit}{Crit}

\DeclareMathOperator{\ad}{ad}

\DeclareMathOperator{\inc}{in}
\DeclareMathOperator{\out}{out}

\DeclareMathOperator{\Sh}{Sh}

\newcommand{\surj}{\twoheadrightarrow}
\newcommand{\inj}{\hookrightarrow}

\newcommand{\pt}{\text{pt}}

\newcommand{\C}{\bbC}

\newcommand{\Q}{\bbQ}
\newcommand{\N}{\bbN}

\DeclareMathOperator{\fac}{fac}
\DeclareMathOperator{\pr}{pr}

\DeclareMathOperator{\BM}{BM}

\newcount\cols
{\catcode`,=\active\catcode`|=\active
 \gdef\Young(#1){\hbox{$\vcenter
 {\mathcode`,="8000\mathcode`|="8000
  \def,{\global\advance\cols by 1 &}%
  \def|{\cr
        \multispan{\the\cols}\hrulefill\cr
        &\global\cols=2 }%
  \offinterlineskip\everycr{}\tabskip=0pt
  \dimen0=\ht\strutbox \advance\dimen0 by \dp\strutbox
  \halign
   {\vrule height \ht\strutbox depth \dp\strutbox##
    &&\hbox to \dimen0{\hss$##$\hss}\vrule\cr
    \noalign{\hrule}&\global\cols=2 #1\crcr
    \multispan{\the\cols}\hrulefill\cr%
   }
 }$}}
}

\begin{document}
\title{The cohomological Hall algebras of a preprojective algebra with symmetrizer}
\date{\today}

\author[Y.~Yang]{Yaping~Yang}
\address{The University of Melbourne,
School of Mathematics and Statistics,
813 Swanston Street, Parkville VIC 3010,
Australia}
\email{yaping.yang1@unimelb.edu.au}

\author[G.~Zhao]{Gufang~Zhao}
\address{The University of Melbourne,
School of Mathematics and Statistics,
813 Swanston Street, Parkville VIC 3010,
Australia}
\email{gufangz@unimelb.edu.au}


\maketitle
\begin{abstract}
This paper aims at a geometric realization of the Yangian of non-simply laced type in terms of quiver with potentials. 
For every quiver with symmetrizer, there is an extended quiver with superpotential, whose Jacobian algebra is the generalized preprojective algebra of Gei{\ss}, Leclerc, and Schr\"oer \cite{GLS}. We study the cohomological Hall algebra of Kontsevich and Soibelman associated to this quiver with potential. In particular, we prove a dimensional reduction result, and provide a shuffle formula of this cohomological Hall algebra. In the case when the quiver with symmetrizer comes from a symmetrizable Cartan matrix, we prove that this shuffle algebra satisfies the relations of the Yangian associated to this Cartan matrix. 
\end{abstract}

\section{Introduction}
For a Kac-Moody Lie algebra $\fg$, the Yangian  of $\fg$ is an affine type quantum algebras associated to $\fg$. In the case when $\fg$ is simply-laced, a geometric construction of Yangian is given using Borel-Moore homology of Nakajima quiver variety \cite{Var} based on earlier works of Nakajima \cite{Nak99}. The present paper aims to construct geometrically the Yangian of non-simply laced type, which was previously unknown. The key observation is that these algebras are obtained in the framework of  Kontsevich-Soibelman cohomological Hall algebras (COHA) for a quiver with potential, when the potential is so that 
the Jacobian algebra is the generalized preprojective algebra of Gei{\ss}, Leclerc, and Schr\"oer \cite{GLS}. The latter is known as a quiver with symmetrizer, hence the title of the present paper. This can be viewed as the first step in studying cohomology of non-simply laced Nakajima quiver variety.

Let  $Q=(I,H)$ be a quiver, with  $I$ the set of vertices, and $H$ the set of arrows. A \textit{symmetrizer} of the quiver $Q$ is a collection of positive integers
\[
L:=\{l_{ij}\in \mathbb{N} \mid i, j\in I, \text{and there is an arrow from $i$ to $j$}\}. 
\]
Associated to the quiver with symmetrizer $(Q, L)$, Gei{\ss}, Leclerc, and Schr\"oer defined a generalized preprojective algebra \cite{GLS}. 

Let $\widehat Q$ be the extended quiver of $Q$.
The set of vertices of  $\widehat Q$ is $I$, and the set of arrows is $H\sqcup H^{\op}\sqcup B$ with $H^{\op}$ in bijection with $H$, and for each $h\in H$, the corresponding arrow in $H^{\op}$, denoted by  $h^*$, is $h$ with orientation reversed. The set 
$B$ is $\{B_i\mid i\in I\}$, with $B_i$ an edge loop at the  vertex $i \in I$. Consider
the following potential of $\widehat Q$ 
\begin{equation*}
W^L:=\sum_{h\in H} \big(B_{\inc(h)}^{l_{\inc(h), \out(h)}}hh^*-B_{\out(h)}^{l_{\out(h), \inc(h)}}h^* h\big). 
\end{equation*}
The generalized preprojective algebra associated to $(Q, L)$  is the Jacobian algebra of $(\widehat Q, \tr(W^L))$ \cite{GLS}. 
The representation variety of the Jacobian algebra of $(\widehat Q, \tr(W^L))$ is the critical locus of $\tr W^L$. This variety has been studied in detail in {\it loc. cit.}.

The Jacobian algebra is naturally endowed with the structure of a DG-algebra \cite{Ginz}. 
In particular, the representation variety is naturally endowed with a complex of constructible sheaves, which is the vanishing cycle of the potential function $\tr(W^L)$. The cohomology of the representation variety valued in this vanishing cycle complex has the structure of an algebra, constructed by Kontsevich and Soibelman \cite{KS}, called the cohomological Hall algebra (COHA).
In the present paper, we study this cohomological Hall algebra.

In the case when the quiver with symmetrizer comes from a symmetrizable Cartan matrix,  
we prove that a localized form of the COHA gives the positive part of the Yangian associated to the corresponding symmetrizable Kac-Moody algebra. 
The precise statement is given as Theorem~\ref{thm:Yangian to sh}. 
The proof uses similar calculations as in \cite[\S~7]{YZ1} and \cite{YZ2}, where the latter significantly relies on \cite[Appendix~A]{D}. Along the way, we give a more general version of dimensional reduction result describing the COHA using Borel-Moore homology instead of critical cohomology. This version is useful in other settings (see, e.g., \cite{RSYZ2}).  We include a sketch of the proofs to make the present paper self-contained.

When the Cartan matrix is symmetric, and the symmetrizer $\{l_{ij}=l\}$ has the same order $l \geq 1$, the COHA is expected to be related to an $l$th zastava space from the work of Mirkovi\'c \cite[\S~3.4]{Mirk14}.This expectations is based on the  comparison of the tautological line bundle on Grassmannian and the sperical COHA in the $l=1$ case. 
The COHA, as well as the spherical subalgebra, are both coherent sheaves on the $I$-colored configuration space of points on $\bbC$, i.e., the moduli of finite subschemes of $\bbC \times I$. The $I$-colored configuration space has a map to the loop Grassmannian of the adjoint group via the Abel-Jacobi map. The
restriction of the tautological line bundle $\calO(1)$ via this map defines a line bundle with a \textit{locallity or factorization structure} \cite[Proposition 3.5.4]{MYZ}.  
The sheaf of sections of this line bundle is identified with the spherical COHA \cite[Proposition 3.5.2]{MYZ}, without torus equivariant parameters. Under this identification, the locality structure of the line bundle corresponds to the  algebra structure of spherical COHA.
 Turning on the equivariant parameters gives a quantization of this line bundle \cite[Section 4]{MYZ}. For general $l\geq 1$, we expect the COHA for general symmetrizer $l$ to give rise to a quantization of the line bundle $\calO(l)$, and therefore lead to a quantization of the homogeneous coordinate ring of the zastava space. 
We pursue this elsewhere. 
 
This quiver with potential and its relation to loop Grassmannian originates from the work of Nekrasov and Shatashvili \cite[\S~3.1.4]{NS}. 
There are already some mathematical work regarding the Jacobian algebra for a quiver with a symmetrizer \cite{Cec,CD}. 
We further expect that the cohomology of the moduli space of stable framed representations of this quiver with potential, valued in the vanishing cycle, gives a global Weyl module of the Yangian. This is the analogue of  Nakajima quiver variety for non-simply laced types. 
In the simply laced with equal symmetrizer  case, notably in the  work  of
Bykov and Zinn-Justin \cite{BZ}, the equivariant cohomology of moduli space of stable framed representations of this Jacobian algebra has been related to, in the $\fs\fl_2$-case, the $(l+1)$-th spin representation. 

\subsection*{Acknowledgements} This paper is motivated by a project in collaboration with Ivan Mirkovi\'c on the quantization of the homogeneous coordinate ring of the zastava space \cite{MYZ}, and is largely inspired by a talk of Paul Zinn-Justin in the Representation Theory Seminar at the University of Melbourne. The idea of doing dimensional reduction in the non-loop directions is suggested to us by Hiraku Nakajima. The authors thank the referee for a careful reading of the manuscript and suggestions to improve the paper. 

The two authors are partially supported by the Australian Research Council via the awards DE190101231 and  DE190101222 respectively.

\section{The cohomological Hall algebra}
In this section, we first review the cohomological Hall algebra (COHA) associated to a quiver with potential that was defined by Kontsevich and Soibelman \cite[Section 7]{KS} (see also \cite{D}). 
We then introduce the extended quiver with potential that is of particular interest in this paper. 
\subsection{The COHA}
Let $\Gamma=(\Gamma_0, \Gamma_1)$ be a quiver, where $\Gamma_0$ is the set of vertices, and $\Gamma_1$ the set of arrows. 
For each $h\in \Gamma_1$, let $\inc(h)$ be the incoming vertex and $\out(h)$ be the outgoing vertex. 
We denote the path algebra of $\Gamma$  by $\C \Gamma$. 
Let $W$ be a potential of $\Gamma$, that is, $W=\sum_{u} c_u u$ with $c_{u} \in \C$, and $u$'s are cycles in $\Gamma$. 

Given a cycle $u=a_1\dots a_n$ (here the cycles are considered up to cyclic order) and an arrow $a\in \Gamma_1$. The cyclic derivative is defined to be
\[
\frac{\partial u}{\partial a}=\sum_{i: a_i=a} a_{i+1}\dots a_n a_1\dots a_{i-1} \in \C \Gamma
\] as an element of $\C \Gamma$. We extend the cyclic derivative to the potential by linearity.

For any dimension vector $v=(v^i)_{i\in \Gamma_0}\in \bbN^{\Gamma_0}$, the representation space of $\Gamma$ with dimension vector $v$ is denoted by $\bold{M}_{\Gamma, v}$. That is, let $V=\{V^i\}_{i\in \Gamma_0}$ be a $|\Gamma_0|$-tuple of vector spaces so that $\dim(V^i)=v^i$.  Then, 
\[
\bold{M}_{\Gamma, v}:=\bigoplus_{h\in \Gamma_1}\Hom(V^{\out(h)},V^{\inc(h)}).
\]
The group $G_{v}:=\prod_{i\in \Gamma_0}\GL(V^i)$ acts on the representation space $\bold{M}_{\Gamma, v}$ via conjugation. 

For a quiver with potential $(\Gamma, W)$ and dimension vector $v\in \bbN^{\Gamma_0}$,  denote by
$\tr(W)_v$ the function on $\bold{M}_{\Gamma, v}$ given by the trace of the potential. Let $\Crit(\tr W_v)$ be the critical locus of $\tr W_v$. 
Let $\varphi_{\tr W_{v}}$ be the vanishing cycle complex on $\bold{M}_{\Gamma, v}$ associated to the function $\tr(W)_v$, which is supported on $\Crit(\tr W_v)$.
We refer the readers to \cite{KS} and \cite{D} for the definition and details of the  vanishing cycle complex. 
For a $G_v$-variety $X$, denote by $H_{c, G_v}^*(X)$ the equivariant cohomology with compact support. 
Let 
 $H_{c, G_v}^*(X)^\vee$ be its dual. 
 Similar for cohomology valued in a complex of sheaves.
We have an isomorphism 
$H_{c, G_v}^*(\bold{M}_{\Gamma, v}, \varphi_{\tr W_v})\cong H_{c, G_v}^*(\Crit(\tr W_v), \varphi_{\tr W_v})$.

Let
\[
\calH(\Gamma, W):=\bigoplus_{v\in \N^{\Gamma_0}}\calH(\Gamma, W)_v
= \bigoplus_{v\in \N^{\Gamma_0}}H_{c, G_v}^*(\Crit(\tr W_v), \varphi_{\tr W_v})^\vee .\] 
There is an algebra structure on $\calH(\Gamma, W)$ via Hall multiplications  \cite[Section 7.6]{KS}.  
The algebra $\calH(\Gamma, W)$ is called the COHA associated to the quiver with potential $(Q, W)$. 
 For a review of the COHA in the presence of a {\it cut}, see also \cite[\S~1.1, 1.2]{YZ2} and \cite{D}.

\subsection{The extended quiver with symmetrizer}\label{subsec:weights}
For a quiver $Q=(I,H)$, with  $I$ the set of vertices, and $H$ the set of arrows, 
we consider the extended quiver $\widehat Q$ as defined in the introduction.
\Omit{
Here the set of vertices of  $\widehat Q$ is $I$, and its set of arrows is $H\sqcup H^{\op}\sqcup B$ with $H^{\op}$ in bijection with $H$, and for each $h\in H$, the corresponding arrow in $H^{\op}$, denoted by  $h^*$, is $h$ with orientation reversed. The set 
$B$ is $\{B_i\mid i\in I\}$, with $B_i$ an edge loop at the  vertex $i \in I$.}

For any $i, j\in I$, denote the set of arrows of $Q$ from $i$ to $j$ by $\{i\to j\}$. 
A collection of integers
\[
L:=\{l_{ij}\in \mathbb{N} \mid i, j\in I, \text{and} \{i\to j\} \neq \emptyset\}
\] 
is called a symmetrizer of the quiver $Q$. 

For a quiver $Q$ with a symmetrizer $L$, we define
a potential of $\widehat Q$ to be
\begin{equation}\label{WL}
W^L:=\sum_{h\in H} \big(B_{\inc(h)}^{l_{\inc(h), \out(h)}}hh^*-B_{\out(h)}^{l_{\out(h), \inc(h)}}h^* h\big). 
\end{equation}
The cyclic derivatives of $W^L$ are
\begin{align}
&\frac{\partial W^L}{ \partial h}=h^* B_{\inc(h)}^{l_{\inc(h), \out(h)}}-B_{\out(h)}^{l_{\out(h), \inc(h)}}h^* \label{eq:derivative1}\\ 
&\frac{\partial W^L}{ \partial h^*}=B_{\inc(h)}^{l_{\inc(h), \out(h)}}h-hB_{\out(h)}^{l_{\out(h), \inc(h)}}\label{eq:derivative2}\\
&\frac{\partial W^L}{ \partial B_i}=\sum_{\{h\mid \inc(h)=i, \out(h)=j\}} \sum_{e=0}^{l_{i,j}-1} B_i^{l_{ij}-1-e}hh^* B_i^{e}
-\sum_{\{h\mid \out(h)=i, \inc(h)=j\}} \sum_{e=0}^{l_{i,j}-1} B_i^{l_{ij}-1-e}h^*h B_i^{e}\label{eq:derivative3}
\end{align}

We consider a {\it weight function} of $Q$: 
\begin{align*}
&\bfm: I \sqcup H\sqcup H^{\op} \to \bbZ, i\mapsto \bfm_i, h\mapsto \bfm_h, h^*\mapsto \bfm_{h^*}, \text{for $i\in I$, $h\in H$ and $h^*\in H^{\op}$}. 
\end{align*}
For each $v\in \bbN^I$, on the space $\bfM_{\widehat Q, v}$, in addition to the action of $\GL_v$, there is a 3-dimensional torus $(\Gm)^3$ action. 
Let $(z_1,z_2,z_3)$ be the  coordinates of $(\Gm)^3$, 
\[
(z_1,z_2,z_3)(h, h^*, B_i)=(z_1^{\bfm_h}h,  z_2^{\bfm_{h^*}}h^*, z_3^{\bfm_i}B_i).  
\]
so that $z_1$ and $z_2$ scale $h$ and $h^*$ respectively with weights $\bfm_h$ and ${\bfm}_{h^*}$   for each $h\in H$, and $z_3$ scales $B_i$ for each $i\in I$ with weight $\bfm_i$. 

The action of $(z_1,z_2,z_3)\in (\Gm)^3$ preserves the potential function $W^L$, if and only if 
\begin{equation}\label{assump1}
z_1^{\bfm_h}z_2^{\bfm_{h^*}} z_3^{\bfm_{\inc(h)} l_{\inc(h), \out(h)}}
=z_1^{\bfm_h}z_2^{\bfm_{h^*}}z_3^{\bfm_{\out(h)} l_{\out(h), \inc(h)}}
=1. 
\end{equation}
In particular, 
\begin{equation}\label{assump2}
\bfm_{i} l_{ij}=\bfm_{j} l_{ji}, \text{for any $i, j\in I$}
\end{equation}
In the present paper, we make the assumption that there exists such a function $\bfm$, such that the condition \eqref{assump2} holds. 

Let $\calD$ be a torus endowed with a group homomorphism $a: \calD\to (\Gm)^3$  such that any element $(z_1, z_2, z_3)$ in the image of $a$ satisfies the condition \eqref{assump1}. 
Thus, the torus $\calD$ acts on $\bfM_{\widehat Q, v}$ in a way which preserves the potential function $W^L$.

We now give an example of one choice of $\calD$ and $\{\bfm_h, \bfm_{h^*}, \bfm_{i}\mid i\in I, h\in H\}$ that satisfy the above assumptions. 
\begin{example}\label{ex:weightFunc}
Let $(Q, L)$ be a quiver with symmetrizer. Assume here $Q$ has no oriented cycles.  
Let $\{\bfm_i\in \mathbb{Z}, i\in I\}$ be a set of integers, such that the condition \eqref{assump2} holds. 
For a pair of vertices $i,j\in I$, denote  the number $\bfm_{i} l_{ij}=\bfm_{j} l_{ji}$ by $d$. 
Let $n$ be the number of arrows between $i$ and $j$ (only in one direction because of the acyclicity assumption) in $Q$. We fix a numbering $h_1, \cdots,  h_{n}$ of these arrows, and let 
\[
\bfm_{h_p}:=(n + 2 - 2p)d , \,\ \bfm_{h_p^*}:=(-n  + 2p)d, \,\ \text{for}\,\  p=1, \cdots, n. 
\]
Consider the embedding $\calD:=\Gm\inj (\Gm)^3, z\mapsto (z, z, z^{-2})$.
The assumption \eqref{assump1} follows from the equality ${\bfm_h} +{\bfm_{h^*}}-2\bfm_{\inc(h)} l_{\inc(h), \out(h)}=2d-2d=0$. This torus action will be used in the construction of Yangians associated to symmetrizable Cartan matrices in \S~\ref{sec:Yangian}.

\Omit{
Following Nakajima's construction of quantum affine algebra associated to $Q$, we use $\calD=\bbG_m$, the diagonal torus in $T=\bbG_m^2$  (see \cite[(2.7.1), (2.7.2)]{Nak99}). This satisfies Assumption~\eqref{eqn:AssumptionWeight}. 
Note that  the set of weights on the arrows between $s,t\in I$ in the extended quiver $\widehat Q$ is independent of the orientation chosen at the beginning. Moreover, if $(Q,L)$ corresponds to $(A,D)$ with level $d$, then the number $n=\frac{a}{d}$ is also equal to $\frac{-{a_{st}}}{l_s}=\frac{-{a_{ts}}}{l_t}$.
}
\end{example}
With $\calD$ satisfying condition \eqref{assump1}, the Hall multiplication of Kontsevich and Soibelman is equivariant with respect to this $\calD$-action. Therefore, we have the following equivariant COHA
\begin{equation}\label{eq:CoHA}
\calH_{\calD}(\widehat Q, W^L)=\bigoplus_{v\in \N^{I}}H_{c, G_v\times \calD}^*(\Crit(\tr W^L_v), \varphi_{\tr W_v})^\vee .
\end{equation}

\section{The shuffle algebra}
In this section, we define the shuffle algebra associated to the COHA $\calH(\widehat Q,  W^L)$. 
In the shuffle algebra considered in this section, there are quantization parameters $t_1, t_2, t_3$. 
Geometrically these parameters come from the torus $\calD\inj (\Gm)^3$ action on representation space of the quiver $\widehat Q$.

\subsection{The definition}
\label{shuffle}
To begin with, we fix some notations. 
Let $(Q,L)$ be a quiver with symmetrizer. We fix a weight function $\bfm: I \sqcup H\sqcup H^{\op} \to \bbZ$.
Let $v=(v^i)_{i\in I} \in \bbN^I$ be a dimension vector of $Q$ and $\fS_v:=\prod_{i\in I} \fS_{v^i}$ be the product of symmetric groups. 
There is a natural action of $\fS_v$ on the variables $\{ \lambda^i_s\mid  i\in I, s=1,\dots, v^i\}$ by permutation. 

For any pair $(p, q)$ of positive integers, let $\Sh(p,q)$ be the subset of $\fS_{p+q}$ consisting of $(p,q)$-shuffles (permutations of $\{1, \cdots, p+q\}$ that preserve the relative order of $\{1,\cdots ,p\}$ and $\{p+1,\cdots ,p+q \}$).
Given dimension vectors $v_1=(v_1^i)_{i\in I}, v_2=(v_2^i)_{i\in I}$ with $v=v_1+v_2$, let  $\Sh(v_1,v_2)\subset \fS_{v}$ denote the product $\prod_{i\in I}\Sh(v_1^i,v_2^i)$. 

We now define the shuffle algebra $\calS\calH$ associated to the data $(Q, L, \bfm)$. 
The shuffle algebra $\calS\calH$ is an $\bbN^I$-graded $\bbC[t_1,t_2, t_3]$-algebra. As a $\bbC[t_1,t_2, t_3]$-module, we have $\calS\calH=\bigoplus_{v\in\bbN^I}\calS\calH_v$. The degree $v$ piece is
\[\calS\calH_v:=\bbC[t_1,t_2, t_3]
[\lambda^i_s]_{i\in I, s=1,\dots, v^i}^{\fS_v}.\] 
We will also consider specializations of the equivariant parameters, i.e., algebraic homomorphisms  $\bbC[t_1,t_2, t_3]\to A$, for some algebra $A$. 
We assume the weight function $\bfm$  is compatible with the specialization in the sense that 
\begin{equation}\label{eqn:AssumptionWeight}
t_1{\bfm_h}+t_2{\bfm_{h^*}}+t_3{\bfm_{\inc(h)} l_{\inc(h), \out(h)}}
=t_1{\bfm_h}+t_2{\bfm_{h^*}}+t_3{\bfm_{\out(h)} l_{\out(h), \inc(h)}}
=0, \text{for any $h\in H$}. 
\end{equation}
Identifying $\bbC[t_1,t_2, t_3]$ with $H_{\Gm^3}(\pt)$, then in the presence of $a:\calD\to \Gm^3$, the specialization $\bbC[t_1,t_2, t_3]\to A$ can be taken to be $H_{\Gm^3}(\pt)\to H_{\calD}(\pt)=:A$ induced by $a$. Condition \eqref{eqn:AssumptionWeight} on the specialization is then equivalent to the condition \eqref{assump1} on $a$.

For any $v_1$ and $v_2\in \bbN^I$, we consider $\calS\calH_{v_1}\otimes \calS\calH_{v_2}$ as a $\bbC[t_1,t_2, t_3]$-submodule of \[\bbC[t_1,t_2, t_3][\lambda^i_j]_{i\in I, j=1,\dots, (v_1+v_2)^i}\] by sending $\lambda'^i_s$ to $\lambda^i_s$, and $\lambda''^i_t$ to $\lambda^i_{t+v_1^i}$. 
Here $\{\lambda'^i_s \mid {i\in I, s=1,\dots, v_1^i}\}$ and $\{\lambda''^i_s \mid {i\in I, s=1,\dots, v_2^i}\}$ are the variables of $\calS\calH_{v_1}$ and  $\calS\calH_{v_2}$ respectively. 
Define\footnote{The $\fac_1$ in the present paper and $\fac_1$ in \cite[\S 3.1]{YZ1} differ by a sign $(-1)^{\sum_{i\in I}v_1^iv_2^i}$. 
The shuffle formula in the present paper is deduced from a 3-dimensional COHA, while the one in \cite{YZ1} is obtained from a 2-dimensional COHA. 
This sign naturally occurs when comparing the dimensional reduction of a 3d COHA to a 2d COHA \cite[\S 5.1]{YZ2}.}
\begin{equation}\label{equ:fac1}
\fac_1:=\prod_{i\in I}\prod_{s=1}^{v_1^i}
\prod_{t=1}^{v_2^i}\frac{\lambda\rq{}\rq{}^i_t-\lambda\rq{}^i_s+\bfm_i t_3}{\lambda\rq{}\rq{}^i_t-\lambda\rq{}^i_s}. 
\end{equation}
and \begin{small}\begin{equation}
\label{equ:fac2}
\fac_2:=\prod_{h\in H}\Big(
\prod_{s=1}^{v_1^{\out(h)}}
\prod_{t=1}^{v_2^{\inc(h)}}
(\lambda_t^{'' \inc(h)}-\lambda_s^{'\out(h)}+ \bfm_h  t_1)
\prod_{s=1}^{v_1^{\inc(h)}}
\prod_{t=1}^{v_2^{\out(h)}}
(\lambda_t^{''\out(h)}-\lambda_s^{'\inc(h)}+\bfm_{h^*} t_2)
\Big).
\end{equation}\end{small}

The multiplication of $f_1(\lambda')\in \calS\calH_{v_1}$ and $f_2(\lambda'')\in \calS\calH_{v_2}$ is defined to be
\begin{equation}\label{shuffle formula}
f_1(\lambda')* f_2(\lambda''):=\sum_{\sigma\in\Sh(v_1,v_2)}\sigma(f_1\cdot f_2\cdot \fac_1\cdot \fac_2)\in \bbC[t_1,t_2, t_3][\lambda^i_j]_{i\in I, j=1,\dots, (v_1+v_2)^i}^{\fS_{v_1+v_2}}, 
\end{equation}
It is a direct algebraic computation that $\calS\calH$ endowed with the above multiplication is an associative algebra. 
Note that although $\fac_1$ has a denominator, the symmetrization over shuffle elements creates zeros on the numerator which cancels the poles introduced by the denominators. Therefore, the shuffle product is well-defined without introducing any localization.

\subsection{COHA and the shuffle algebra} \label{sec:shuffle}
In this section, we take into account the torus $\calD$-action on the representation space of $\widehat{Q}$. We compute $\calH_{\calD}(\widehat Q, W^L)$ in terms of the shuffle algebra in \S\ref{shuffle}. The main ingredient is the dimension reduction of Davison \cite{D} and \cite[Theorem 2.5]{YZ2} recalled in Appendix~\ref{sec:ReviewDimRed}. Here we follow the notations in Appendix~\ref{sec:ReviewDimRed}. 

We take $\Gamma$ to be $\widehat{Q}$, the potential to be $W^L$.    
We take the cut to be $H$, the set of arrows of the original quiver $Q$. 
Let $\widehat{Q}\setminus H$ be the quiver obtained from $\widehat{Q}$ by removing $H$. 
Hence the set of vertices of  $\widehat{Q}\setminus H$ is $I$, and its set of arrows is $H^{\op}\sqcup B=\{h^*, B_i\mid h^*\in H^{\op}, i\in I\}$. Consider the quotient of the path algebra $\bbC(\widehat{Q}\setminus H)$ by the relation 
\[
h^* B_{\inc(h)}^{l_{\inc(h), \out(h)}}=B_{\out(h)}^{l_{\out(h), \inc(h)}}h^*, \,\ \text{for any $h\in H$. }
\]
By \eqref{eq:derivative1}, the representation variety of this quotient algebra is then 
\begin{align*}
\bold{J}_{\widehat{Q}\setminus H, v}
:&= \{x \in \bold{M}_{\widehat{Q}\setminus H, v} \mid \partial W^L/ \partial h (x)=0,  \,\ \text{for any $h\in H$}\}
\\
&=\{(h^*, B_i)\in \bold{M}_{\widehat{Q}\setminus H, v} \mid h^* B_{\inc(h)}^{l_{\inc(h), \out(h)}}=B_{\out(h)}^{l_{\out(h), \inc(h)}}h^*, \,\ \text{for any $h\in H$}
 \}.
 \end{align*}
\begin{remark}\label{iso:loc}
Let $i: \bold{J}_{\widehat{Q}\setminus H, v} \times \bold{M}_{Q, v}\inj \bold{M}_{\widehat{Q}, v}$ be the natural embedding. Pushforward along $i$ gives \[
i_*:H^{\BM}_{G_v\times\calD}(\bold{J}_{\widehat{Q}\setminus H, v} \times \bold{M}_{Q, v}, \bbQ)\to
 H^{\BM}_{G_v\times\calD}(\bold{M}_{\widehat{Q}, v}, \bbQ),\] which is an isomorphism after localization. More precisely, take $T_v\subseteq G_v$ to be a maximal torus. Let $T=T_v\times\calD$. Restricting to $T$ gives a commutative diagram 
 \[
\xymatrix{
H^{\BM}_{G_v\times\calD}(\bold{J}_{\widehat{Q}\setminus H, v} \times \bold{M}_{Q, v}, \bbQ)\ar[r]^(0.6){i_*}\ar[d]&
 H^{\BM}_{G_v\times\calD}(\bold{M}_{\widehat{Q}, v}, \bbQ)\ar[d]\\
 H^{\BM}_{T}(\bold{J}_{\widehat{Q}\setminus H, v} \times \bold{M}_{Q, v}, \bbQ)\ar[r]^(0.6){i_*}&
 H^{\BM}_{T}(\bold{M}_{\widehat{Q}, v}, \bbQ)}
\]
The bottom row is a map of modules over $H_T^{\BM}(\pt)\cong \bbC[\Lie T]$. Let $\bbC(\Lie T)$ be the quotient field of $\bbC[\Lie T]$. By the equivariant localization \cite{GKM98}, the bottom map $i_*$  induces an isomorphism when base changed from $\bbC[\Lie T]$ to $\bbC(\Lie T)$. 
\end{remark}
In Appendix~\ref{sec:ReviewDimRed}, Theorem \ref{thm:Hall}, we show $\calH_{\calD}(\widehat Q, W^L)$ is canonically isomorphic to $ \oplus_{v}H^{\BM}_{G_v\times\calD}(\bold{J}_{\widehat{Q}\setminus H, v} \times \bold{M}_{Q, v}, \bbQ)$. 
\begin{theorem}\label{thm:shuffle}
The pushforward $i_*$ induces an algebra homomorphism $\calH_{\calD}(\widehat Q, W^L)\to \calS\calH$. It is an isomorphism after localization in the sense above.
\end{theorem}

\begin{proof}
With notations as in Appendix~\ref{sec:ReviewDimRed},
the first row of Diagram~\eqref{diag:j23} becomes 
\begin{equation*}
\xymatrix@R=1.5em   {
\bold{M}_{\widehat{Q}, v_1}\times \bold{M}_{\widehat{Q}, v_2} & 
(\bold{M}_{\widehat{Q} \backslash H, v_1}\times \bold{M}_{\widehat{Q} \backslash H, v_2})\oplus
\bold{M}_{Q, v_1, v_2}
\ar[l]_(0.6){p_1}\ar[r]^{i_1}
& 
(\bold{M}_{\widehat{Q} \backslash H, v_1}\times \bold{M}_{\widehat{Q} \backslash H, v_2}) \oplus 
\bold{M}_{Q, v}}\end{equation*}
The first row of \eqref{corresp with fiber} becomes 
\begin{equation*}
\xymatrix@R=1.5em@C=2.5em{
(\bold M_{\widehat{Q}\setminus H, v_1}) \times (\bold M_{\widehat{Q}\setminus H, v_2})
\ar@{^{(}->}[r]^(0.8){\iota}&\bold{Y}&
\bold M_{\widehat{Q}\setminus H, v_1, v_2} \ar[l]_(0.7){\omega} \ar[r]^{i_2}&
\bold M_{\widehat{Q}\setminus H, v}
}
 \end{equation*}
 with
 \begin{align*}
 \bold{Y}=\{(y_{h^*}, x_{h^*} , B)\mid  y_{h^*}\in &\bold M_{Q^{\op},v_1,v_2},  
 x_{h^*} \in  (\bold M_{Q^{\op},v_1}\times \bold M_{Q^{\op},v_2}), B\in  (\fg\fl_{v_1}\times \fg\fl_{v_2})
   \\& \pr(y_{h^*})=B_{\inc h}^{l_{\inc(h), \out(h)}}x^*_h-x^*_hB_{\out(h)}^{l_{\out(h), \inc(h)}}, \text{for any $h\in H$}\}, 
 \end{align*}
 where $\pr: \bold M_{Q^{\op},v_1,v_2}\to \bold M_{Q^{\op},v_1}\times \bold M_{Q^{\op},v_2}$ is the natural projection. 
Theorem~\ref{thm:Hall} gives an isomorphism of algebras
\[
\calH_{\calD}(\widehat Q, W^L)\cong \bigoplus_{v\in \bbN^{I}} H^{\BM}_{G_v\times\calD}(\bold{J}_{\widehat{Q}\setminus H, v} \times \bold{M}_{Q, v}, \bbQ). 
\]
The Hall multiplication $m^{\crit}$ of the former is identified as 
 \begin{equation}\label{eqn:shuffle_geom0}
(\overline{i_{2}}\times \id_{\bold{M}_{Q,v}} )_*  \circ \frac{1}{e(\iota)}(\omega\times \id_{\bold{M}_{Q,v}})_{\overline{\omega}\times \id_{\bold{M}_{C,v}}}^{\sharp} \circ \overline{i_{1}}_* \circ \overline{p_1}^{*}.
\end{equation}
of the latter. In particular, the latter is associative.

On $\bigoplus_{v\in \bbN^{I}} H^{\BM}_{G_v\times\calD}(\bold{M}_{\widehat{Q}, v}, \bbQ)$ we have the binary operation defined by 
 \begin{equation}\label{eqn:shuffle_geom}
({i_{2}}\times \id_{\bold{M}_{Q,v}} )_*  \circ (\omega\times \id_{\bold{M}_{Q,v}})^{*} \circ \frac{1}{e(\iota)} (\iota\times \id_{\bold{M}_{Q,v}})_{*} \circ  {i_{1}}_* \circ {p_1}^{*}.
\end{equation}
Recall $i: \bold{J}_{\widehat{Q}\setminus H, v} \times \bold{M}_{Q, v}\to \bold{M}_{\widehat{Q}, v}$ is the natural embedding. 
Pushforward along $i$ for each $v$ gives 
\[i_*:
\bigoplus_{v\in \bbN^{I}} H^{\BM}_{G_v}(\bold{J}_{\widehat{Q}\setminus H, v} \times \bold{M}_{Q, v}, \bbQ)\to
\bigoplus_{v\in \bbN^{I}} H^{\BM}_{G_v}(\bold{M}_{\widehat{Q}, v}, \bbQ),
\] 
which intertwines the above binary operations \eqref{eqn:shuffle_geom0} and \eqref{eqn:shuffle_geom}. 

Now we identify the target with $\calS\calH$, and \eqref{eqn:shuffle_geom} with \eqref{shuffle formula}. This follows from the same calculation  
of the Thom classes of $\iota$, $i_1$, and $i_2$ as in \cite[\S~1.4]{YZ1}. 
More precisely, the pushforward $\iota_*$ is the same as multiplication by $e(\iota)$. 
 The normal bundle to $i_1$ and $i_2$ are identified with  
 \begin{align*}
\sHom_{Q}(\calR(v_1), \calR(v_2))\,\ \text{and} \sHom_{\widehat{Q}\setminus H} (\calR(v_1), \calR(v_2))
\end{align*}
respectively, where $\calR(r)$ is the tautological bundle of $\Grass(r, \infty)$.
Let $e(i)$ be the equivariant Euler class of the normal bundle of $i$ and 
let $\{\lambda_{s}^i\}_{s=1, \dots, v^i}$ be the Chern roots of the tautological bundle $\calR(v^i)$. Thus, we have
\begin{align*}
e(i_1)e(i_2)=&\prod_{i\in I}\prod_{s=1}^{v_1^i}
\prod_{t=1}^{v_2^i}(\lambda\rq{}\rq{}^i_t-\lambda\rq{}^i_s+\bfm_i t_3)\\
\cdot &\prod_{h\in H}\Big(
\prod_{s=1}^{v_1^{\out(h)}}
\prod_{t=1}^{v_2^{\inc(h)}}
(\lambda_t^{'' \inc(h)}-\lambda_s^{'\out(h)}+ \bfm_h  t_1)
\prod_{s=1}^{v_1^{\inc(h)}}
\prod_{t=1}^{v_2^{\out(h)}}
(\lambda_t^{''\out(h)}-\lambda_s^{'\inc(h)}+\bfm_{h^*} t_2)
\Big), 
\end{align*}
 The denominator of $\fac_1$ and the averaging over all the shuffle $\Sh(v_1,v_2)$ come from pushing-forward from a Grassmannian bundle.  
 
In the above we see one more time that \eqref{eqn:shuffle_geom} is well-defined without introducing denominators. The operation \eqref{eqn:shuffle_geom0} is associative, and hence remains so after localization. On the other hand, the natural map from $H^{\BM}_{G_v\times\calD}(\bold{M}_{\widehat{Q}, v}, \bbQ)$ to its localization is injective. In particular, \eqref{eqn:shuffle_geom} is associative after localization, and hence is so before localization. This shows that $\calS\calH$ is an algebra.  The argument above then implies that $i_*$ is an algebra homomorphism.
This completes the proof of Theorem \ref{thm:shuffle}. 

\end{proof}

\begin{remark} We expect a version of $\bold{J}_{\widehat{Q}\setminus H, v}$ with framing can be defined, together with suitable stability conditions, so that the cohomology groups of the stable framed representations carry representations of the double of the COHA. 
\end{remark}

\section{Generalized Cartan matrix}
\label{sec:Cartan matrix}
In this section, we consider a symmetrizable Cartan matrix. We associate to 
a Cartan matrix with symmetrizer a quiver with symmetrizer. 

\subsection{The potential associated to the Cartan matrix}
Let $A = (a_{ij} )_{1\leq i,j \leq  n}$ be a symmetrizable generalized Cartan matrix.
Thus, $a_{ii} = 2$ for all $1\leq i \leq n$;  $a_{ij} \leq 0$ for any $1\leq i \neq j \leq n$; $a_{ij}\neq 0$ if and only if $a_{ji}\neq 0$, 
and there exists
a diagonal matrix $D=\diag(d_1, \cdots, d_n)$ with positive integer entries such that $DA$ is
symmetric. In other words, we have $d_ia_{ij}=d_ja_{ji}$ for any $1\leq i, j\leq n$. 
For simplicity, we drop the words "symmetrizable generalized" and  call the pair $(A,D)$ a {\it Cartan matrix with symmetrizer}. 

Start with a Cartan matrix with symmetrizer $(A, D)$. We construct a quiver with symmetrizer as follows. 
The set of vertices is $I=\{1, 2, \cdots, n\}$. For any $i<j\in I$, the number of arrows from $i$ to $j$ is
\[
|\gcd(a_{ij}, a_{ji})|. 
\]
We choose the symmetrizer of the quiver to be 
\[
l_{ij}:=|\frac{a_{ij}}{\gcd(a_{ij}, a_{ji})}|. 
\]
Thus, there exist  weights $
\bold{m}_i=d_i$ such that $\bold{m}_i l_{ij}=\bold{m}_j l_{ji} $, since $DA$ is symmetric.

\Omit{
For a positive integer $d$, we say that $(A,D)$ and $(Q,L)$ {\it correspond} to each other with {\it level of correspondence} $d$,  if for any $i\in I$ we have \[d_il_i=d\] and for any $s\neq t\in I$ we have 
\[\#\{s\to t\}+\#\{t\to s\}=\frac{-a_{st}}{l_s}=\frac{-a_{ts}}{l_t}.\]

In other words, let $C$ be the symmetric matrix 
with $c_{st}=d_sa_{st}$ for any $s\neq t\in I$. 
Let $(ad_{st})_{s,t\in I}$ be adjacency matrix of $Q$, and $\overline{(ad)}$ the transpose or equivalently the adjacency matrix of the opposite quiver $Q^{\op}$. Then,  $ad_{st}+ad_{ts}=\frac{-c_{st}}{d}$ if $s\neq t$, and $Q$ has no self-loops.}

\begin{example}
When the Cartan matrix is $A=\begin{pmatrix} 2 & -1\\ -2 & 2\end{pmatrix}$. We have $d_1=\bold{m}_1=2$ and $d_2=\bold{m}_2=1$. In this case, the symmetrizer of the quiver is $l_{12}=1$ and $l_{21}=2$. 
\end{example}
The following remark is kindly pointed out to us by Hiraku Nakajima.
\begin{remark}
When the Cartan matrix $A$ in the present paper is taken to be the transpose of the Cartan matrix in \cite{GLS}, 
the quiver just defined with potential $W^{L}=\sum_{h\in H} \big(B_{\inc(h)}^{l_{\inc(h), \out(h)}}hh^*-B_{\out(h)}^{l_{\out(h), \inc(h)}}h^* h\big)$ 
agrees with the quiver with potential in \cite[\S 1.7.3]{GLS}, with $l_{ij}$ in the present paper equal to $f_{ij}$ of {\it loc. cit.} for any $i,j\in I$. 
The vanishing of the cyclic derivatives \eqref{eq:derivative1}, \eqref{eq:derivative2}, \eqref{eq:derivative3} gives the relations (P2) and (P3) of \cite[Section 1.4]{GLS}.  
In particular, the Yangian constructed in \S~\ref{sec:Yangian} is the Yangian of the Langlands dual to the Lie algebra constructed in {\it loc cit.}.

Starting from the same quiver with potential in the present paper, the Lie algebra associated to it in  \cite{NW, HL}  is Langlands dual to that of \cite{GLS}, 
thus it is the same Lie algebra as in the present paper. 
\end{remark}
\Omit{
Note that if $Q$ has self-loop, then such a Cartan matrix with symmetrizer does not exist for any $d$. Conversely, for each $d$, the quiver with symmetrizer $(Q,L)$ might not be unique even if exists, although $Q$ and $L$ individually is unique if exists given the other. }

\subsection{A sign twist}
\label{subsec:twisting CoHA}
We now work under the setting of Example~\ref{ex:weightFunc}. By specializing $t_1=t_2=\hbar/2$, and $t_3=-\hbar$, 
we have a shuffle algebra $\calS\calH$ over $\bbC[\hbar]$.  
We define a sign-twisted shuffle algebra $\widetilde{\calS\calH}$. As a $\bbC[\hbar]$-module, $\widetilde{\calS\calH}$ is the same as $\calS\calH$. 
The weight function $\bold{m}$ is given as in Example~\ref{ex:weightFunc}. 
In particular, we have
$\bold{m}_i=d_i$, for $i=1, 2, \cdots, n$.

Define
\begin{align}
\widetilde{\fac}_1:&=\fac_1
=\prod_{i\in I}\prod_{s=1}^{v_1^i}
\prod_{t=1}^{v_2^i}\frac{\lambda\rq{}^i_s-\lambda\rq{}\rq{}^i_t-\bfm_i t_3}{\lambda\rq{}^i_s-\lambda\rq{}\rq{}^i_t} \label{fac1tw}\\
\widetilde{\fac}_2:&=(-1)^{\sum_{h\in H} v_1^{\inc(h)}v_2^{\out(h)}} \fac_2  \notag\\
&=\prod_{h\in H}\Big(
\prod_{s=1}^{v_1^{\out(h)}}
\prod_{t=1}^{v_2^{\inc(h)}}
(\lambda_t^{'' \inc(h)}-\lambda_s^{'\out(h)}+ \bfm_h  t_1)
\prod_{s=1}^{v_1^{\inc(h)}}
\prod_{t=1}^{v_2^{\out(h)}}
(\lambda_s^{'\inc(h)}-\lambda_t^{''\out(h)}-\bfm_{h^*} t_2)
\Big)\label{fac2tw}
\end{align}
For simplicity, we write the multiplication of $\widetilde{\calS\calH}$ as $\star$. 
The multiplication of $f_1(\lambda')\in \widetilde{\calS\calH}_{v_1}=\calS\calH_{v_1}$ and $f_2(\lambda'')\in \widetilde{\calS\calH}_{v_2}=\calS\calH_{v_2}$ is defined to be
\begin{equation}\label{shuffletw}
f_1(\lambda')\star f_2(\lambda'')=(-1)^{\sum_{h\in H} v_1^{\inc(h)}v_2^{\out(h)}}f_1(\lambda')* f_2(\lambda'')=
\sum_{\sigma\in\Sh(v_1,v_2)}\sigma(f_1\cdot f_2\cdot \widetilde{\fac}_1\cdot \widetilde{\fac}_2)\in \calS\calH_{v_1+v_2}. 
\end{equation}

\section{Shuffle presentation of a Yangian for symmetrizable Cartan matrix}
\label{sec:Yangian}
\subsection{The Yangian }
Let $\fg_{A,D}$ be the symmetrizable Kac-Moody Lie algebra associated to the Cartan matrix $A=(a_{ij})_{i, j\in I}$ with symmetrizer $D$. In this section we assume the collection of integers $(d_i)_{i\in I}$ to be relatively prime. This assumption is rather a conventional choice of normalization in the presentation of the Yangian we use \cite{GTL}.
Let $(Q,L)$ be the quiver with symmetrizer corresponding to $\fg_{A,D}$. Without raising confusions, we will also write $\fg_{A,D}$ as $\fg_Q$.

The Kac-Moody presentation of $\mathfrak{g}_Q$ is given by $\fg=\langle e_i, f_i \mid i\in I\rangle$ such that
\begin{align*}
&(\ad e_i )^{1-a_{ij}} e_j=0=(\ad f_i)^{1-a_{ij}} f_j, \,\  [e_i, f_j]=0, i\neq j,  \\
&[[e_i, f_i], e_j]=a_{ij} e_j,  [[e_i, f_i], f_j]=-a_{ij} f_j, i, j\in I. 
\end{align*}

Recall that the Yangian of $\fg_Q$, denoted by  $Y_\hbar(\fg_Q)$, is an associative algebra over $\bbQ[\hbar]$, generated by the variables
\[
x_{k, r}^{\pm}, h_{k, r}, (k\in I, r\in \N),
\]
subject to certain relations. A complete list of relations can be found in \cite{GTL}, which  we refrain from recalling in this paper. 
Relevant to us, define $Y_\hbar^+(\fg_Q)$ to be the quotient of the free algebra on the generators $x_{k, r}^{+}$, for $k\in I, r\in \N$ by the following relations. 
Define the generating series $x_k^+(u)\in Y_{\hbar}^+(\fg_Q)[\![u^{-1}]\!]$ by
$x_k^+(u)=\hbar \sum_{r\geq 0} x_{k, r}^+ u^{-r-1}. $
The following is a complete set of relations defining $Y_\hbar^+(\fg_Q)$:
\begin{align}
&(u-v-d_ka_{kl}\hbar/2) x_k^+(u)x_l^+(v)
=(u-v+d_ka_{kl}\hbar/2) x_l^+(v)x_k^+(u)
 \tag{Y1}\label{calY1} \\ &\phantom{ABCDEFGABCDEFG}+\hbar\Big(
[x_{k, 0}^+, x_{l}^+(v)]-[x_{k}^+(u), x_{l, 0}^+]
\Big), \text{for any $k, l\in I$}. \notag\\
&
\sum_{\sigma\in\fS_{1-a_{kl}}}[x_k^+(u_{\sigma(1)}),[x_k^+(u_{\sigma(2)}),[\cdots,[x_k^+(u_{\sigma (1-a_{kl})}),x_l^+(v)]\cdots]]]=0, \text{for $k \neq l\in I$}.
\tag{Y2}\label{calY2}
\end{align}
Denote by $Y_{\hbar}^{sub}(\fg_Q)$ the subalgebra of $Y_{\hbar}(\fg_Q)$ generated by $x_{k, r}^{+}$, for $k\in I, r\in \N$. It is clear that there is a surjective homomorphism $Y_{\hbar}^{+}(\fg_Q)\surj Y_{\hbar}^{sub}(\fg_Q)$, and they may differ in general. 
\Omit{
\begin{example}
Let $\fg=\mathfrak{sl}_4$ be of type $A_3$ with Cartan matrix 
$\begin{pmatrix}
2&-1 & 0\\
-1&2 & -1\\
0&-1 & 2
\end{pmatrix}$. 
The simple Lie algebra $\mathfrak{sp}_4$ is of type $C_2$ with Cartan matrix $\begin{pmatrix}
2&-1\\
-2&2
\end{pmatrix}$ ($d_1=2$ and $d_2=1$ in this case). 
It is well-known that the Dynkin diagram of $\mathfrak{sp}_4$ can be obtained by folding the 
Dynkin diagram of $\mathfrak{sl}_4$. For Lie algebras, the folding is the Lie algebra homomorphism
\[
\mathfrak{sp}_4 \to \mathfrak{sl}_4, e_{2}' \mapsto e_{1}+e_{3}, e_{1}' \mapsto e_{2}. 
\]
Denote by $Y(A_3)$ the Yangian associated to $\mathfrak{sl}_4$. We write the generators of $Y^+(A_3)$ as $\{ x_{1, r}, x_{2, r}, x_{3, r}\mid r\in \mathbb{N}\}$. Similarly, denote by $Y(C_2)$ the Yangian associated to  $\mathfrak{sp}_4$. The generators of $Y^+(C_2)$ is denoted by $\{ x'_{1, r}, x'_{2, r}\mid r\in \mathbb{N}\}$. 
Unlike in the case of the Lie algebra, the obvious assignments $x_{2, r}' \mapsto x_{1, r}+x_{3, r}, x_{1, r}' \mapsto x_{2, r}$ does not extend to an algebra homomorphism from the Yangian $Y^+(C_2)$ to $Y^+(A_3)$. The quadratic relation 
\[
[x_{2, r+1}', x_{2, s}']-[x_{2, r}', x_{2, s+1}']=-\hbar (x_{2, r}'x_{2, s}'+x_{2, s}'x_{2, r}') 
\] is not preserved. Algebraically, it is not clear how to get the non-simply laced Yangian via folding. 
\end{example}
}
\subsection{Main result}
In this section, we compare the Yangian of non-simply laced type with the COHA $\calH_{\calD}(\widehat Q, W^L)$ associated to quiver with symmetrizer. 

Let $\widetilde{\calS\calH}^{\sph} $ be the spherical subalgebra of $\widetilde{\calS\calH}$ generated by $ \calS\calH_{e_k}$ as $k$ varies in $I$. 
\begin{theorem}\label{thm:Yangian to sh}
The assignment
\[Y^+_\hbar(\fg_Q)\ni  x_{k, r}^+ \mapsto (\lambda^{(k)})^{r}\in \calS\calH_{e_k}=\bbC[\hbar][\lambda^{(k)}]\]
 extends to a well-defined algebra epimorphism $Y^+_{\hbar}(\fg_Q)\to\widetilde{\calS\calH}^{\sph}|_{t_1=t_2=\hbar/2, t_3=-\hbar}$. 
\end{theorem}

\begin{remark}
The proof of Theorem \ref{thm:Yangian to sh} is similar as that of \cite[Theorem~7.1]{YZ1}, taking into account the shuffle formulas in the present setting. The main difference here, compared to the proof in \cite{YZ1} are the values of $n,d,a$ from  Example \ref{ex:weightFunc}, the definitions of $S$ and $S'$.
Note also that the Serre relation \eqref{calY2} is asymmetric with respect to $k$ and $l$. See 
Remark~\ref{rmk:Serre} for a detailed discussion. We only include the reduction of the Serre relation to a form to which \cite[Corollary~A2]{YZ1} can be directly applied.

\end{remark}

In order to prove Theorem~\ref{thm:Yangian to sh}, we need to verify  the  relations \eqref{calY1} and \eqref{calY2}  in the algebra $\widetilde{\calS\calH}$. We include here the detailed computation parallel to the proof of \cite[Theorem~7.1]{YZ1} to keep track of $D=\diag(d_1, \cdots, d_n)$ from the Yangian side and the symmetrizer $L$ from the COHA side. 

\subsubsection{The quadratic relation \eqref{calY1}}
We now check the relation \eqref{calY1} in the shuffle algebra.
We have
\[
x_k^+(u)\mapsto \hbar \sum_{r\geq 0} (\lambda^{(k)})^r u^{-r-1}=\frac{\hbar}{u-\lambda^{(k)}}, 
\]
where the equality means the expansion of the rational function $\frac{\hbar}{u-\lambda^{(k)}}$ at $u=\infty$. 

To check the  quadratic relation \eqref{calY1}, we need to show
\begin{align}
(u-v-\frac{\hbar d_k a_{kl}}{2}) 
&\frac{\hbar}{u-\lambda^{(k)}}
\star\frac{\hbar}{v-\lambda^{(l)}}
-(u-v+\frac{\hbar d_k a_{kl}}{2}) 
\frac{\hbar}{v-\lambda^{(l)}}\star
\frac{\hbar}{u-\lambda^{(k)}} \label{equ:quad}\\
=&\hbar
\Big(1^{(k)}\star\frac{\hbar}{v-\lambda^{(l)}}
-\frac{\hbar}{v-\lambda^{(l)}}\star 1^{(k)}
-\frac{\hbar}{u-\lambda^{(k)}}\star 1^{(l)}
+ 1^{(l)}\star\frac{\hbar}{u-\lambda^{(k)}}
\Big).\notag
\end{align}

We first consider the case when $k\neq l$. We spell out the formula of the multiplication
$\widetilde{\calS\calH}_{e_k} \otimes \widetilde{\calS\calH}_{e_l}\to \widetilde{\calS\calH}_{e_k+e_l}$ as a map 
$\bbC[\hbar][\lambda^{(k)}]\otimes  \bbC[\hbar][\lambda^{(l)}] \to \bbC[\hbar][\lambda^{(k)}, \lambda^{(l)}]$. 
Plugging-in $v_1=e_k$, and $v_2=e_l$ to \eqref{fac1tw},  we have $\widetilde{\fac}_1=1$. 
If there is no arrow between $k$ and $l$, then both sides of \eqref{equ:quad} are zero. 
Without loss of generality, we assume there are $n$ arrows from $k$ to $l$. 
As in Example \ref{ex:weightFunc}, we have
\begin{align*}
& n=|\gcd(a_{kl}, a_{lk})|, \,\ d=\bold{m}_k l_{kl}=|\frac{d_k a_{kl}}{\gcd(a_{kl}, a_{lk})}|, 
\text{thus } \,\  a:=nd=-d_k a_{kl}=-d_l a_{lk}, \\
&\text{and the weights $\bold{m}$ are}: 
\bfm_{h_p}:=a+2d-2pd \,\ \bfm_{h_p^*}:=-a  + 2pd, \,\ \text{for}\,\  p=1, \cdots, n. 
\end{align*}
 Let $S$ be the set of integers $\{a, a-2d, a-4d, \dots, -a+4, -a+2d\}$. 
 Then, $S=\{ \bfm_{h_p}\mid 1\leq p\leq n\}=\{ \bfm_{h_p^*}\mid 1\leq p\leq n\}$. 
Set $S':=\{a-2d, a-4d, \dots, -a+2d\}$, we then have 
\begin{equation}\label{S'S}
S'\sqcup \{a\}=S, \text{and} \,\ S'=-S'. 
\end{equation}
 Plugging the weight function $\bold{m}$ into  \eqref{fac2tw}, we have $\widetilde{\fac}_2=\prod_{m\in S}(\lambda^{(l)}-\lambda^{(k)}+m\frac{\hbar}{2})$. Therefore, by the shuffle formula \eqref{shuffletw}, the Hall multiplication is given by
\[
(\lambda^{(k)})^p \star(\lambda^{(l)})^q=
(\lambda^{(k)})^p (\lambda^{(l)})^q \prod_{m\in S}(\lambda^{(l)}-\lambda^{(k)}+m\frac{\hbar}{2}).
\] 
Similarly, the multiplication
$\widetilde{\calS\calH}_{e_l} \otimes \widetilde{\calS\calH}_{e_k}\to \widetilde{\calS\calH}_{e_k+e_l}$
is given by
\begin{align*}
(\lambda^{(l)})^p \star(\lambda^{(k)})^q
=(\lambda^{(l)})^p (\lambda^{(k)})^q \prod_{m\in S}(\lambda^{(l)}-\lambda^{(k)}-m\frac{\hbar}{2}).
\end{align*}
Plugging  into equation \eqref{equ:quad}, \eqref{calY1} becomes the following identity
\begin{align}
(u-v-\frac{\hbar d_k a_{kl}}{2}) &
\frac{\hbar}{u-\lambda^{(k)}}
\frac{\hbar}{v-\lambda^{(l)}}
\prod_{m\in S}(\lambda^{(l)}-\lambda^{(k)}+m\frac{\hbar}{2}) \notag\\
-(u-v+\frac{\hbar d_k a_{kl}}{2}) &
\frac{\hbar}{v-\lambda^{(l)}}
\frac{\hbar}{u-\lambda^{(k)}}\prod_{m\in S}(\lambda^{(l)}-\lambda^{(k)}-m\frac{\hbar}{2})  \notag\\
=&\hbar
\Big(\frac{\hbar}{v-\lambda^{(l)}}
\prod_{m\in S}(\lambda^{(l)}-\lambda^{(k)}+m\frac{\hbar}{2})
-\frac{\hbar}{v-\lambda^{(l)}}
\prod_{m\in S}(\lambda^{(l)}-\lambda^{(k)}-m\frac{\hbar}{2}) \label{equ*}\\
&-\frac{\hbar}{u-\lambda^{(k)}}\prod_{m\in S}(\lambda^{(l)}-\lambda^{(k)}+m\frac{\hbar}{2})
+\frac{\hbar}{u-\lambda^{(k)}}
\prod_{m\in S}(\lambda^{(l)}-\lambda^{(k)}-m\frac{\hbar}{2})
\Big).  \notag
\end{align}

Using \eqref{S'S}, we have the common factor 
\[
\hbar^2\prod_{m\in S'}(\lambda^{(l)}-\lambda^{(k)}+m\frac{\hbar}{2})=\hbar^2\prod_{m\in S'}(\lambda^{(l)}-\lambda^{(k)}-m\frac{\hbar}{2}). 
\] 
Canceling the above common factor, the equality \eqref{equ*} becomes
\begin{align*}
&(u-v+\frac{\hbar a}{2}) 
\frac{\lambda^{(l)}-\lambda^{(k)}+a\frac{\hbar}{2}}{(u-\lambda^{(k)})(v-\lambda^{(l)})}-(u-v-\frac{\hbar a}{2}) 
\frac{\lambda^{(l)}-\lambda^{(k)}-a\frac{\hbar}{2}}{(v-\lambda^{(l)})(u-\lambda^{(k)})}
=
a\hbar\Big( \frac{1}{v-\lambda^{(l)}}
-\frac{1}{u-\lambda^{(k)}}
\Big). 
\end{align*}
Both sides of the above identity are equal to $a\hbar\frac{u-v+\lambda^{(l)}-\lambda^{(k)}}{(u-\lambda^{(k)})(v-\lambda^{(l)})}$. 
This shows  the relation \eqref{calY1} for the case when $k\neq l$. 

We now check the relation \eqref{calY1} when $k=l$.
A similar calculation using the shuffle formula \eqref{shuffletw} shows that equation \eqref{equ:quad} becomes the following identity in $\widetilde{\calS\calH}_{2e_k}
=\bbC[\hbar][\lambda_1, \lambda_2]$
\begin{align*}
&(u-v-d_k\hbar) \sum_{\sigma\in\fS_2}\sigma \Big(
\frac{\hbar}{u-\lambda_1}
\frac{\hbar}{v-\lambda_2}
\frac{\lambda_1-\lambda_2+ d_k\hbar}{\lambda_1-\lambda_2}\Big)
-(u-v+d_k\hbar) \sum_{\sigma\in\fS_2}\sigma
\Big(\frac{\hbar}{v-\lambda_1}
\frac{\hbar}{u-\lambda_2}\frac{\lambda_1-\lambda_2+ d_k\hbar}{\lambda_1-\lambda_2}\Big)\\
&=\hbar\sum_{\sigma\in\fS_2}\sigma
\Bigg(\Big(\frac{\hbar}{v-\lambda_2}
-\frac{\hbar}{v-\lambda_1}
-\frac{\hbar}{u-\lambda_1}
+\frac{\hbar}{u-\lambda_2}\Big)
\frac{\lambda_1-\lambda_2+ d_k\hbar}{\lambda_1-\lambda_2}
\Bigg).
\end{align*} 
It is straightforward to show that both sides of the above identity can be simplified to 
\[\frac{2d_k\hbar^3}{\lambda_1-\lambda_2} \Big(\frac{1}{v-\lambda_2}
-\frac{1}{v-\lambda_1}
-\frac{1}{u-\lambda_1}
+\frac{1}{u-\lambda_2}\Big).\]
This completes the proof of relation \eqref{calY1} for $k=l$.

\subsubsection{The Serre relation \eqref{calY2}}
By an argument similar to \cite[\S~10.4]{Nak99},  to show \eqref{calY2}, without loss of generality it suffices to check
\[\sum_{p=0}^{1-a_{kl}} (-1)^p
{1-a_{kl} \choose p} x_{k, 0}^{\star p} 
\star x_{l, 0} \star x_{k, 0}^{\star(1-a_{kl}-p)}=0, 
\tag{Y2$'$}\label{calY2'}\]
where $x^{\star n}=x\star x\star \cdots \star x$, the shuffle product of $n$-copies of $x$.
We use the shuffle formula \eqref{shuffletw} to check the Serre relation \eqref{calY2'}.

For any $i,j$, let $\lambda_{i,j}=\lambda_i-\lambda_j$. By the shuffle formula \eqref{shuffletw}, we have the recurrence relation
\begin{align*}
&(x_{k, 0})^{\star n+1}
=\sum_{\sigma \in\Sh_{(n,1)}} \sigma\Big( (x_{k, 0})^{\star n}  \prod_{i=1}^n \frac{\lambda^{(k)}_{i, n+1}+d_k\hbar}{\lambda^{(k)}_{i, n+1}}\Big).
\end{align*}
Therefore, inductively, we get a formula of $(x_{k, 0})^{\star n}$:
\begin{align}
(x_{k, 0})^{\star n}
=&\sum_{\sigma\in \fS_n}
\sigma\left(
 \frac{\lambda^{(k)}_{12}+d_k\hbar}{\lambda^{(k)}_{12}}\cdot
 \frac{\lambda^{(k)}_{13}+d_k\hbar}{\lambda^{(k)}_{13}}
 \cdots
 \frac{\lambda^{(k)}_{n-1, n}+d_k\hbar}{\lambda^{(k)}_{n-1, n}}\right). \label{1*}
\end{align}

Note that $k\neq l$. By the shuffle formula \eqref{shuffletw}, the multiplication
$\widetilde{\calS\calH}_{n e_k} \otimes \widetilde{\calS\calH}_{e_l}\to\widetilde{\calS\calH}_{n e_k+e_l}
$ is given by
\begin{align}
&(x_{k, 0})^{\star n}\star x_{l, 0}
=(x_{k, 0})^{\star n} \prod_{i=1}^n \prod_{m\in S}
(\lambda^{(l)}-\lambda^{(k)}_i+m\frac{\hbar}{2}), \label{2*}
\end{align} where $S=\{a, a-2d, a-4d, \dots, -a+4d, -a+2d\}$. Here again we write $a=-d_ka_{kl}$.

For the multiplication $\widetilde{\calS\calH}_{p e_k+ e_l} \otimes \widetilde{\calS\calH}_{q e_k}\to \widetilde{\calS\calH}_{e_l+(p+q) e_k}$, considered as a map \[
\bbC[\hbar ][\lambda^{(k)}_1, \cdots,\lambda^{(k)}_p, \lambda^{(l)}_{1} ] \otimes \bbC[\hbar ][\lambda^{(k)}_{p+1}, \cdots,\lambda^{(k)}_{p+q}]  \to 
\bbC[\hbar][\lambda^{(k)}_1, \cdots,\lambda^{(k)}_{p+q}, \lambda^{(l)}_{1}],\] we have
\begin{align}
((x_{k, 0})^{\star p}\star x_{l, 0})\star(x_{k, 0})&^{\star q}
=\sum_{\sigma \in \Sh_{(p, q)}}\sigma \Big((x_{k, 0})^{\star p}\star x_{l, 0}\cdot  (x_{k, 0})^{\star q} \cdot \notag \\
&\cdot  \prod_{s=1}^p \prod_{t=p+1}^{p+q}
\frac{\lambda^{(k)}_s-\lambda^{(k)}_t+d_k\hbar }{\lambda^{(k)}_s-\lambda^{(k)}_t} 
\cdot \prod_{t=p+1}^{p+q} \prod_{m\in S}
(\lambda^{(l)}-\lambda^{(k)}_t -m\frac{\hbar}{2})\Big).
 \label{3*}
\end{align}
Plugging the formulas of \eqref{1*} \eqref{2*} into  \eqref{3*} with $q=1-a_{kl}-p$, we get
\begin{align*}
x_{k, 0}^{\star p} 
\star x_{l, 0} \star x_{k, 0}^{\star(1-a_{kl}-p)}&
=\sum_{\pi\in \Sh_{(p, 1-a_{kl}-p)}}\pi \Bigg( 
\Big( \sum_{\sigma\in \fS_p} \sigma\cdot \prod_{1\leq i<j \leq p}
\frac{\lambda^{(k)}_{i, j}+d_k\hbar}{\lambda^{(k)}_{i, j}}\Big)\\
&\cdot \Big( \sum_{\sigma\in \fS_{1-a_{kl}-p}} \sigma \cdot \prod_{\{p+1\leq i<j \leq 1-a_{kl}\}}
\frac{\lambda^{(k)}_{i, j}+d_k\hbar}{\lambda^{(k)}_{i, j}}\Big)
\Big(\prod_{s=1}^p \prod_{t=p+1}^{1-a_{kl}}
\frac{\lambda^{(k)}_{s, t}+d_k\hbar }{\lambda^{(k)}_{s, t}}\Big)\\
&\cdot 
\prod_{m\in S}\Big( \prod_{i=1}^p(\lambda^{(l)}-\lambda^{(k)}_i+m\frac{\hbar}{2}) 
\prod_{t=p+1}^{1-a_{kl}} 
(\lambda^{(l)}-\lambda^{(k)}_t -m\frac{\hbar}{2})\Big)\Bigg).
\end{align*}
Re-arranging the above summation, we have:
\begin{align*}
\sum_{p=0}^{1-a_{kl}} (-1)^p 
{1-a_{kl} \choose p} 
x_{k, 0}^{\star p} 
\star x_{l, 0} \star & x_{k, 0}^{\star(1-a_{kl}-p)}
=\sum_{p=0}^{1-a_{kl}} (-1)^p
{1-a_{kl} \choose p} \sum_{\sigma\in \fS_{1-a_{kl}}}
\Big( \prod_{1\leq i<j \leq 1-a_{kl}}
\frac{\lambda^{(k)}_{\sigma(i), \sigma(j)}+d_k\hbar}{\lambda^{(k)}_{\sigma(i), \sigma(j)}}\cdot 
\\
&
\cdot \prod_{m\in S}\Big( \prod_{i=1}^p(\lambda^{(l)}-\lambda^{(k)}_{\sigma(i)}+m\frac{\hbar}{2}) 
\prod_{t=p+1}^{1-a_{kl}} 
(\lambda^{(l)}-\lambda^{(k)}_{\sigma(t)} -m\frac{\hbar}{2})\Big)\Big).
\end{align*}

Note that the factor
\[
\prod_{m\in S'}\Big( \prod_{i=1}^p(\lambda^{(l)}-\lambda^{(k)}_{\sigma(i)}+m\frac{\hbar}{2}) 
\prod_{t=p+1}^{1-a_{kl}} 
(\lambda^{(l)}-\lambda^{(k)}_{\sigma(t)} -m\frac{\hbar}{2})\Big)
=\prod_{m\in S'}\prod_{i=1}^{1-a_{kl}}
(\lambda^{(l)}-\lambda^{(k)}_{i}-m\frac{\hbar}{2}),
\]  is independent of $\sigma\in \mathfrak{S}_{1-a_{kl}}$, hence a common factor.  Here again $S'=\{a-2d, a-4d, \dots, -a+2d\}$.
Let $\lambda'^{(k)}_i=\lambda^{(k)}_i-\lambda^{(l)}$. 
After canceling the above common factor, to show the Serre relation \eqref{calY2'}, it suffices to show
\begin{align}
&\sum_{p=0}^{1-a_{kl}} (-1)^p
{1-a_{kl} \choose p} \sum_{\sigma\in \fS_{1-a_{kl}}}
 \sigma\left( 
 \prod_{s=1}^p(\lambda'^{(k)}_{s}-\frac{a\hbar}{2}) 
\prod_{t=p+1}^{1-a_{kl}} 
(\lambda'^{(k)}_{t}+\frac{a\hbar}{2})
\prod_{1\leq i<j \leq 1-a_{kl}}
\frac{\lambda^{(k)}_{i, j}+
d_k\hbar}{\lambda^{(k)}_{i, j}} \right)=0.\label{simply:serre}
\end{align}
The identity \eqref{simply:serre} is \cite[Corollary~A2]{YZ1}. This proves \eqref{calY2}.

\begin{remark}\label{rmk:Serre}
Note that for non-simply laced Cartan matrix $A$, when $a_{kl} \neq a_{lk}$, the Serre relation \eqref{calY2} is asymmetric switching $k$ and $l$. 
This is reflected in the fact that in \eqref{simply:serre}, the last factor 
$\prod_{1\leq i<j \leq 1-a_{kl}}\frac{\lambda^{(k)}_{i, j}+
d_k\hbar}{\lambda^{(k)}_{i, j}} $ is different switching $k$ and $l$. In addition, although the number $a$ is the same, the decompositions $a=-a_{kl}d_k$ and $a=-a_{lk}d_l$ used in the proof above are different  switching $k$ and $l$.
\end{remark}

As a consequence of Theorem \ref{thm:shuffle} and Theorem \ref{thm:Yangian to sh}, we obtain an algebra epimorphism
\[
\Phi: Y^+_\hbar(\fg_{A, D}) \surj (\calH_{\calD}(\widehat Q, W^L)/\text{tors})^{\sph} \cong \widetilde{\calS\calH}^{\sph}, 
\]
where $\text{tors}$ means the torsion elements of $\calH_{\calD}(\widehat Q, W^L)$ which are defined to be elements in the kernel of the morphism $\calH_{\calD}(\widehat Q, W^L)\to \calS\calH$. 
Assume the quiver $Q$ is of finite type and simply laced,  it follows from the work of Schiffmann and Vasserot \cite[Theorem A (c)]{SV17a} that $\calH_{\calD}(\widehat Q, W^L)_v$ is torsion free over $H^{BM}_{\calD\times \GL_v}(\pt)$. The Yangian action on the cohomology of quiver varieties is faithful \cite[Lemma 8.1]{YZ1}.  The argument in the proof of \cite[Theorem 8.3]{YZ1} shows that the morphism $\Phi$ is an isomorphism.
For general quiver $Q$, after tensoring the fractional field $K$ of $H^{BM}_{\calD}(\pt)$, we have $\calH_{\calD}(\widehat Q, W^L)_v\otimes_{H^{BM}_{\calD}(\pt)} K$ is torsion free over $H^{BM}_{\calD\times \GL_v}(\pt)\otimes_{H^{BM}_{\calD}(\pt)} K$ \cite[Theorem B (d)]{SV17a}. The comparison of the Yangian of Maulik and Okounkov and COHA can be found in \cite{SV17b}. 
 
We do not know if either statement holds in general.  To investigate these questions, 
we need to consider a faithful representation of the Yangian, which is expected to be obtained from the cohomology of Nakajima quiver varieties in the present setting, generalizing the construction of \cite{Var}. In the present paper, we do not consider frames or stablity conditions of representations of the quiver with symmetrizer, hence do not address these questions. 

\appendix
\section{Review of dimensional reduction}
\label{sec:ReviewDimRed}
We review a dimensional reduction procedure that describes the cohomological Hall algebra in the presence of a cut. 
Such a dimensional reduction on the level of cohomology groups is obtained by Davison \cite[Appendix~A]{D}, which is the main ingredient here. 

In this section we study the behaviour of the algebra structure under the dimensional reduction, which is needed in the present paper. 
The argument here  is similar to that in the proof of 
 \cite[Theorem~2.5]{YZ2}. However, the  statement here is more general and more useful (see e.g., \cite[Section 7.1]{RSYZ2}). 
 We include a brief sketch of the argument in present generality, highlighting the difference to that in  \cite[Theorem~2.5]{YZ2}.

Let $\Gamma=(\Gamma_0, \Gamma_1)$ be a quiver, and $W$ be the potential. 
In general, 
a cut $C$ of $(\Gamma, W)$ is a subset $C\subset \Gamma_1$ such that $W$ is homogeneous of degree 1 with respect to the grading defined on arrows by

\begin{displaymath} 
\deg a= 
\left\{
     \begin{array}{lr}
       1 & : a \in C, \\
       0 & : a \notin C.
     \end{array}
   \right.
   \end{displaymath}
Consider the quotient of the path algebra $\C (\Gamma \backslash C)$ by the relations 
$\{{\partial W}/{\partial a} \mid a\in C \}$. The representation variety of this quotient algebra is denoted by
\[
\bold{J}_{\Gamma\backslash C, v}:=\{ x\in \bold{M}_{\Gamma \backslash C, v} \mid {\partial W}/{\partial a}(x)=0,  \forall a\in C \}.\] 
We view $C$ as forming the edges of a new quiver, which we still denote by $C$ for simplicity. Let $ \bold{M}_{C, v}$ be the representation variety of $C$ with dimension vector $v$. Consider the trivial vector bundle
 $\pi:  \bold{M}_{\Gamma, v}= \bold{M}_{\Gamma \backslash C, v}\times  \bold{M}_{C, v}  \to  \bold{M}_{\Gamma \backslash C, v}$ carrying a scaling $\Gm$ action of weight one on the fiber $\bold{M}_{C, v} $.
Let $\tr W_v:  \bold{M}_{\Gamma, v}= \bold{M}_{\Gamma \backslash C, v}\times \bold{M}_{C, v} \to \mathbb{A}^1$ be the function which is $\Gm$--equivariant. 
Define $Z \subset  \bold{M}_{\Gamma \backslash C, v}$ to be the reduced scheme consisting of points $z\in  \bold{M}_{\Gamma \backslash C, v}$, such that $\pi^{-1}(z)\subset (\tr W_v)^{-1}(0)$. 
Then, we have 
$Z=\{x\in  \bold{M}_{\Gamma \backslash C, v}\mid \tr(W_v) (x, l)=0, \forall l\in \bold{M}_{C, v} \}$. 
To summarize the notations, we have the diagram:
\[
\xymatrix@R=1.5em{
Z \times  \bold{M}_{C, v} \ar@{^{(}->}[r]^{i\times \id}\ar[d]^{\pi_Z} &\bold{M}_{\Gamma \backslash C, v}\times  \bold{M}_{\Gamma \backslash C, v} \ar[d]^{\pi} \ar[dr]&\\
Z \ar@{^{(}->}[r]^{i} &\bold{M}_{\Gamma \backslash C, v}  \ar[r]^{p}& \pt.
}\]

\begin{lemma} \label{lem:Z_J}
The subvariety $Z$ of $\bold{M}_{\Gamma \backslash C, v}$ is naturally identified with $\bold{J}_{\Gamma\backslash C, v}$. 
\end{lemma}
\begin{proof}
The lemma follows from the same proof of \cite[Lemma 3.1]{YZ2}. 
The difference in the current setting is the non-degenerate paring 
\[
\tr(- \cdot -): \Hom(\bbC^{v_{\inc(a)}},\bbC^{v_{\out(a)}})\times \Hom(\bbC^{v_{\out(a)}},\bbC^{v_{\inc(a)}})\to \bbC.
\] 
given by the trace. 
\Omit{
By definition, the potential $W$ is homogeneous of degree $1$. 
This implies 
\begin{equation}\label{equ:trW}
\tr W=\tr \sum_{a\in C} (\partial W/\partial a) a.
\end{equation}
Clearly $\bold{J}_{\Gamma \backslash C, v} \subseteq Z$. 

Conversely, for any $x\in Z$, we have $\tr(W_v) (x, l)=0$, for all $l=(l_a)_{a\in C} \in \bold{M}_{C, v}$. Then, for any $a\in C$, by the equality \eqref{equ:trW} we have $\tr((\partial W/\partial a)(x) \cdot  l_a)=0$, for any matrix $l_a$. 
Note that $\tr(A\cdot l_a)$ defines a non-degenerate paring 
\[
\tr(- \cdot -): \Hom(\bbC^{v_{\inc(a)}},\bbC^{v_{\out(a)}})\times \Hom(\bbC^{v_{\out(a)}},\bbC^{v_{\inc(a)}})\to \bbC.
\] 
This shows the vanishing $(\partial W/\partial a)(x)=0$. Therefore, $x\in \bold{J}_{\Gamma \backslash C, v}$.}
\end{proof}

Let $\calD$ be a torus $\Gm^r$ for some $r\in\bbN$. 
To each arrow in $\Gamma_1$, we associate a $\calD$-weight such that $\tr W_v$ is  $\calD$-invariant for any $v$. 
Then,  $\partial W/\partial a$ is homogeneous for any $a\in C$. In particular, $\bold{J}_{\Gamma\backslash C, v}$ is a $\calD$-equivariant subvariety of $\bfM_{\Gamma\backslash C, v}$.

There is a canonical isomorphism of vector spaces   \cite[Theorem A.1]{D}
\begin{align*}
H_{c, G_v\times
\calD}^*(\bold{M}_{\Gamma, v}, \varphi_{\tr W_v})^{\vee}&
\cong H^{\BM}_{G_v\times
\calD}(\bold{J}_{\Gamma\backslash C, v}\times \bold{M}_{C, v}
, \Q),  \,\ \text{for $v\in \bbN^{\Gamma_0}$}, 
\end{align*}
Following \cite{YZ2}, we now describe a multiplication $m^{\bold{J}}$ on the graded vector space
 \[
 \bigoplus_{v\in \N^{\Gamma_0}} H^{\BM}_{G_v\times
\calD}(\bold{J}_{\Gamma\backslash C, v}\times \bold{M}_{C, v}, \bbQ). 
 \]
  
 Let $v_1,v_2\in \bbN^{\Gamma_0}$ be dimension vectors such that $v=v_1+v_2$, let $V_1\subset V $ be a $|\Gamma_0|$-tuple of subspaces of $V$ with dimension vector $v_1$. 
Define $\bold{M}_{\Gamma, v_1, v_2}:=\{x\in \bold{M}_{\Gamma, v}\mid x(V_1)\subset V_1\}$.  We write $G:=G_{v}\times
\calD$ for short. Let $P\subset G_v\times
\calD$ be the parabolic subgroup preserving the subspace $V_1$ and $L:=G_{v_1}\times G_{v_2}\times
\calD$ be the Levi subgroup of  $P$. We have the following correspondence of $L$-varieties. 
\begin{equation}\label{basic corresp}
\xymatrix
{\bold{M}_{\Gamma, v_1}\times \bold{M}_{\Gamma, v_2}&\bold{M}_{\Gamma, v_1, v_2} \ar[l]_(0.4){p_{\Gamma}}\ar[r]^(0.4){\eta_{\Gamma}} &\bold{M}_{\Gamma, v_1+v_2},
}\end{equation}
where $p_{\Gamma}$ is the natural projection and $\eta_{\Gamma}$ is the embedding. 
\Omit{
Similarly, we have correspondences 
\xymatrix
{\bold{M}_{C, v_1}\times \bold{M}_{C, v_2}& \bold{M}_{C, v_1, v_2}\ar[l]_(0.4){p_{C}}\ar[r]^(0.4){\eta_{C}} &\bold{M}_{C, v},
} and 
$\bold{M}_{\Gamma\backslash C, v_1}\times \bold{M}_{\Gamma\backslash C, v_2} \leftarrow \bold{M}_{\Gamma\backslash C, v_1, v_2} \rightarrow \bold{M}_{\Gamma\backslash C, v}$. }

We have the following commutative diagram
\begin{equation}
\label{diag:j23}
\xymatrix@R=1.5em   {
\bold{M}_{\Gamma, v_1}\times \bold{M}_{\Gamma, v_2} & 
(\bold{M}_{\Gamma \backslash C, v_1}\times \bold{M}_{\Gamma \backslash C, v_2})\times
\bold{M}_{C, v_1, v_2}
\ar[l]_(0.6){p_1}\ar[r]^{i_1}
& 
(\bold{M}_{\Gamma \backslash C, v_1}\times \bold{M}_{\Gamma \backslash C, v_2}) \times 
\bold{M}_{C, v}\\
{
\begin{matrix}
(\bold{J}_{\Gamma \backslash C, v_1}\times \bold{J}_{\Gamma \backslash C, v_2})\\
\times
(\bold{M}_{C, v_1}\times \bold{M}_{C, v_2} )\ar@{^{(}->}[u]
\end{matrix}
}
& 
(\bold{J}_{\Gamma \backslash C, v_1}\times \bold{J}_{\Gamma \backslash C, v_2})\times
\bold{M}_{C, v_1, v_2}
\ar[l]_(0.6){\overline{p_1}}\ar[r]^{\overline{i_1}}\ar@{^{(}->}[u]
& 
(\bold{J}_{\Gamma \backslash C, v_1}\times \bold{J}_{\Gamma \backslash C, v_2}) \times 
\bold{M}_{C, v}\ar@{^{(}->}[u]
}
\end{equation}
where $p_1=\id_{\bold{M}_{\Gamma \backslash C, v_1}\times \bold{M}_{\Gamma \backslash C, v_2}} \times p_{C}$, and $i_1=\id_{\bold{M}_{\Gamma \backslash C, v_1}\times \bold{M}_{\Gamma \backslash C, v_2}} \times \eta_{C}$. 
Here $p_C, \eta_C$ are maps in the correspondence $\bold{M}_{C, v_1}\times \bold{M}_{C, v_2}\xleftarrow{p_{C}} \bold{M}_{C, v_1, v_2} \xrightarrow{\eta_{C}} \bold{M}_{C, v}$. The vertical maps are natural inclusions, and 
$\overline{p_1}, \overline{i_1}$ are the restrictions of $p_1, i_1$. 

Identify $\bold{M}_{C^{\op}, v}$ with $\bold{M}_{C, v}^*$ via $
\bold{M}_{C^{\op}, v}\cong \bold{M}_{C, v}^*, x\mapsto \big(y \mapsto \tr(x  y)\big)$. 
For $x \in \bold{M}_{\Gamma \backslash C, v}$,  the cyclic derivative $\partial W/ \partial a (x)$ is an element in $\bold{M}_{C^{\op}, v}$, for any $a\in C$. 
Thus, for any $l\in \bold{M}_{C, v}$, we have the pairing $(\partial W/ \partial a (x), l)=tr( \partial W/ \partial a (x) \cdot l) $.

Recall that $p_{C^{\op}}: \bold{M}_{C^{\op}, v_1,v_2}\to \bold{M}_{C^{\op}, v_1}\times \bold{M}_{C^{\op}, v_2}$ is the natural projection. 
Introduce the following subvariety
 $\bold{Y}\subset  \bold{M}_{\Gamma \backslash C, v_1}\times \bold{M}_{\Gamma \backslash C, v_2} \times \bold{M}_{C^{\op}, v_1,v_2}$. 
\begin{align}
\bold{Y}:=
&\{(x, l)\mid  x \in \bold{M}_{\Gamma \backslash C, v_1}\times \bold{M}_{\Gamma \backslash C, v_2}, l \in \bold{M}_{C^{\op}, v_1,v_2},   \,\ \text{such that $({\partial W}/{\partial a})_{a\in C}(x)=p_{C^{\op}}(l)$}\}.  \label{eq:varietyY}
\end{align}
There are two maps
\begin{align*}
&\iota: \bold{M}_{\Gamma \backslash C, v_1}\times \bold{M}_{\Gamma \backslash C, v_2} \inj \bold{Y}, \text{given by} \,\ x \mapsto (x, ({\partial W}/{\partial a})_{a\in C}(x)) \\
&\omega: \bold{M}_{\Gamma \backslash C, v_1, v_2} \to Y, \,\   \text{given by} \,\  x\mapsto (p_{\Gamma \backslash C}(x), ({\partial W}/{\partial a})_{a\in C}(x) ).
\end{align*}
Let $\bold{J}_{\Gamma \backslash C, v_1, v_2} \subset \bold{M}_{\Gamma \backslash C, v_1, v_2}$ be the subvariety defined by the equation $({\partial W}/{\partial a})_{a\in C}(x)$, for all  $a\in C$. We then have an embedding $\overline{i_2}: \bold{J}_{\Gamma \backslash C, v_1, v_2}\subset  \bold{J}_{\Gamma \backslash C, v}$. 
They fit into the following commutative diagram.
\begin{equation}\label{corresp with fiber}
\xymatrix@R=1.5em@C=2.5em{
\bold{M}_{\Gamma \backslash C, v_1} \times \bold{M}_{\Gamma \backslash C, v_2}
\ar@{^{(}->}[r]^(0.7){\iota}&\bold{Y}&
\bold{M}_{\Gamma \backslash C, v_1, v_2} 
\ar[l]_(0.6){\omega} \ar[r]^{i_2}&
\bold{M}_{\Gamma \backslash C, v} 
\\
\bold{J}_{\Gamma \backslash C, v_1}\times \bold{J}_{\Gamma \backslash C, v_2}
\ar@{^{(}->}[u] \ar[ur]_{\bar{\iota}}
&&
\bold{J}_{\Gamma \backslash C, v_1, v_2} 
\ar[ll]_{\overline{\omega}}\ar[r]^{\overline{i_2}}\ar@{^{(}->}[u]&
\bold{J}_{\Gamma \backslash C, v} \ar@{^{(}->}[u]
}
\end{equation}
where the map $\overline{\omega}$ is the restriction of $\omega$. Note that by introducing the variety $\bold{Y}$, 
the pullback of the two maps $\omega: \bold{M}_{\Gamma \backslash C, v_1, v_2} \to \bold{Y}$ and $\bar{\iota}: \bold{J}_{\Gamma \backslash C, v_1}\times \bold{J}_{\Gamma \backslash C, v_2}\to \bold{Y}$ is $\bold{J}_{\Gamma \backslash C, v_1, v_2} $. In other words, the square in the diagram \eqref{corresp with fiber} formed by $\omega, \bar{\iota}, \bar{\omega}$ is Cartesian. 
\Omit{In other words, there is a natural isomorphism $$(\bold{J}_{\Gamma\backslash C, v_1}\times \bold{J}_{\Gamma\backslash C, v_2})\times_{\bold{Y}} \bold{M}_{\Gamma\backslash C, v_1, v_2}\cong \bold{J}_{\Gamma\backslash C, v_1, v_2}.$$
Indeed,
\begin{align*}
&(\bold{J}_{\Gamma\backslash C, v_1}\times \bold{J}_{\Gamma\backslash C, v_2})\times_{\bold{Y}} \bold{M}_{\Gamma\backslash C, v_1, v_2}\\
=&
\{
(x_1, x_2)\in \bold{J}_{\Gamma\backslash C, v_1}\times \bold{J}_{\Gamma\backslash C, v_2}, 
x \in \bold{M}_{\Gamma\backslash C, v_1, v_2} \mid 
p_{\Gamma\backslash C}(x)=(x_1, x_2),  ({\partial W}/{\partial a})_{a\in C}(x)=0\}\\
=&\{ x \in \bold{M}_{\Gamma\backslash C, v_1, v_2} \mid 
({\partial W}/{\partial a})_{a\in C}(x)=0\}
=\bold{J}_{\Gamma\backslash C, v_1, v_2}.
\end{align*}}
The multiplication $m^{\bold{J}}$ is defined to be 
\begin{equation}
m^{\bold{J}}:=(\overline{i_{2}}\times \id_{\bold{M}_{C,v}} )_*  \circ \frac{1}{e(\iota)}(\omega\times \id_{\bold{M}_{C,v}})_{\overline{\omega}\times \id_{\bold{M}_{C,v}}}^{\sharp} \circ \overline{i_{1}}_* \circ \overline{p_1}^{*}.
\end{equation}
The maps in the composition are the following. 
\begin{enumerate}
\item The K\"unneth morphism
$
H^{\BM}_{G_{v_1}\times
\calD}(\bold{J}_{\Gamma\backslash C, v_1} \times \bold{M}_{C, v_1})
\otimes H^{\BM}_{G_{v_2}\times
\calD}(\bold{J}_{\Gamma\backslash C, v_2} \times \bold{M}_{C, v_2})
 \to 
H^{\BM}_{L}(\bold{J}_{\Gamma\backslash C, v_1}\times \bold{J}_{\Gamma\backslash C, v_2} \times \bold{M}_{C, v_1}\times  \bold{M}_{C, v_2}). 
$ Here the tensor is over $H_\calD^{\BM}(\pt)$.
\item 
$\overline{i_{1}}_* \circ \overline{p_1}^{*}: 
H^{\BM}_{L}(\bold{J}_{\Gamma\backslash C, v_1}\times \bold{J}_{\Gamma\backslash C, v_2}  \times \bold{M}_{C, v_1}\times  \bold{M}_{C, v_2})\to
H^{\BM}_{L}(\bold{J}_{\Gamma\backslash C, v_1} \times \bold{J}_{\Gamma\backslash C, v_2}  \times \bold{M}_{C, v}). 
$
\item 
Denote by $(\omega\times \id_{\bold{M}_{C,v}})_{\overline{\omega}\times \id_{\bold{M}_{C,v}}}^{\sharp}$ the refined Gysin pullback of $\omega\times\id_{\bold{M}_{C,v}}$ along $\overline{\omega}\times\id_{\bold{M}_{C,v}}$. Let $e(\iota)$ be the  $L$-equivariant Euler class of the normal bundle of $\iota$. 
We have the following map
\[
\frac{1}{e(\iota)}\omega_{\overline{\omega}}^{\sharp}: H^{\BM}_{L}( 
\bold{J}_{\Gamma\backslash C, v_1}\times \bold{J}_{\Gamma\backslash C, v_2} \times \bold{M}_{C, v})
\to H^{\BM}_{L}( \bold{J}_{\Gamma\backslash C, v_1, v_2} \times \bold{M}_{C, v} )[\frac{1}{e(\iota)}]. 
\] 
\item
The pushforward $
(\overline{i_{2}}\times \id_{\bold{M}_{C,v}} )_*: 
H^{\BM}_{L}( \bold{J}_{\Gamma\backslash C, v_1, v_2} \times \bold{M}_{C, v} )\to 
H^{\BM}_{L}( \bold{J}_{\Gamma\backslash C, v} \times \bold{M}_{C, v} )
$. 
\item
Pushforward along 
$G\times_{P} (\bold{J}_{\Gamma\backslash C, v} \times \bold{M}_{C, v})\to \bold{J}_{\Gamma\backslash C, v} \times \bold{M}_{C, v}, (g, m)\mapsto gmg^{-1}$,
 we get
$H^{\BM}_{P}( \bold{J}_{\Gamma\backslash C, v} \times \bold{M}_{C, v} )
\cong 
H^{\BM}_{G}( G\times_P(\bold{J}_{\Gamma\backslash C, v} \times \bold{M}_{C, v} ))
\to 
H^{\BM}_{G}( \bold{J}_{\Gamma\backslash C, v} \times \bold{M}_{C, v} )
$.
\end{enumerate}
This map $m^{\bold{J}}$ {\it a priori} is only defined after inverting $e(\iota)$. However, it follows from Theorem~\ref{thm:Hall} that it is  well-defined before localization. 

\begin{theorem}\label{thm:Hall}
There is an isomorphism of algebras 
\[
\calH_\calD(\Gamma, W)\cong \bigoplus_{v\in \bbN^{\Gamma_0}} H^{\BM}_{G_v\times
\calD}(\bold{J}_{\Gamma\backslash C, v} \times \bold{M}_{C, v}, \bbQ),
\] where $\calH_\calD(\Gamma, W)$ endowed with the Hall multiplication of Kontsevich-Soibelman and $\bigoplus_{v\in \bbN^{\Gamma_0}} H^{\BM}_{G_v\times
\calD}(\bold{J}_{\Gamma\backslash C, v} \times \bold{M}_{C, v}, \bbQ)$ has multiplication given by $m^{\bold{J}}$. 
\end{theorem}
As has been mentioned, a special case of this is \cite[Theorem 2.5]{YZ2} and \cite[Appendix, Corollary 4.5]{RS}.
The proof of  \cite[\S~2]{YZ2} goes through in the setting verbatim, with the following substitutions. 
\begin{enumerate}
\item We need to distinguish between $C$ and $C^{\op}$ in the present paper, where in \cite{YZ2},  $C$ consists of edge loops and hence $C=C^{\op}$. 
\item Consequently we need to identify $\bold{M}_{C^{\op}, v}$ with $\bold{M}_{C, v}^*$ using the trace map. 
\item The variety $\bold{Y}$ is changed to \eqref{eq:varietyY}.
\item The \cite[Lemma 3.1]{YZ2} is replaced by Lemma \ref{lem:Z_J}.  
\item The torus $\calD$-action in the present generality is given in \S~\ref{subsec:weights}. 
\end{enumerate}

\newcommand{\arxiv}[1]
{\texttt{\href{http://arxiv.org/abs/#1}{arXiv:#1}}}
\newcommand{\doi}[1]
{\texttt{\href{http://dx.doi.org/#1}{doi:#1}}}
\renewcommand{\MR}[1]
{\href{http://www.ams.org/mathscinet-getitem?mr=#1}{MR#1}}

\end{document}